\documentclass[a4paper]{article}
\usepackage[all]{xy}\usepackage[latin1]{inputenc}        
\usepackage[dvips]{graphics,graphicx}
\usepackage{amsfonts,amssymb,amsmath,color,mathrsfs, amstext}
\usepackage{amsbsy, amsopn, amscd, amsxtra, amsthm,authblk}
\usepackage{enumerate,algorithmicx,algorithm}
\usepackage{algpseudocode}
\usepackage{upref}
\usepackage{geometry}
\usepackage{subfigure}
\usepackage[displaymath]{lineno}
\usepackage{float}
\usepackage{yhmath}
\usepackage{booktabs}
\usepackage{multirow}
\usepackage{makecell}
\usepackage[colorlinks,
            linkcolor=blue,
            anchorcolor=blue,
            citecolor=blue
            ]{hyperref}

\numberwithin{equation}{section}

\newcommand{\tcb}[1]{\textcolor{blue}{#1}}

\newcommand{\R}{\mathbb{R}}

\newcommand{\cD}{\mathcal{D}}

\newcommand{\cM}{\mathcal{M}}

\newtheorem{theorem}{Theorem}[section]
\newtheorem{lemma}{Lemma}[section]

\newtheorem{remark}{Remark}[section]
\newtheorem{proposition}{Proposition}[section]
\newtheorem{assumption}{Assumption}[section]
\newtheorem{corollary}{Corollary}[section]
\numberwithin{equation}{section}

\begin{document}

\title{Energy and quadratic invariants preserving methods for Hamiltonian systems with holonomic constraints}

\author[1]{Lei Li\thanks{E-mail: leili2010@sjtu.edu.cn}}
\author[2]{Dongling Wang \thanks{E-mail: wdymath@nwu.edu.cn.}}
\affil[1]{School of Mathematical Sciences, Institute of Natural Sciences, MOE-LSC, Shanghai Jiao Tong University, Shanghai, 200240, P. R. China.}
\affil[2]{Department of Mathematics and Center for Nonlinear Studies, Northwest University, Xi'an, Shaanxi, 710127, P. R. China.}
\maketitle

\begin{abstract}
We introduce a new class of parametricization structure-preserving partitioned Runge-Kutta ($\alpha$-PRK) methods for Hamiltonian systems with holonomic constraints. When the scalar parameter $\alpha=0$, the methods are reduced to the usual symplectic PRK methods like Shake-Rattle method or PRK schemes based on Lobatto IIIA-IIIB pairs, which can preserve all the quadratic invariants and the constraints.
When $\alpha\neq 0$, the methods are also shown to preserve all the quadratic invariants and the constraints manifold exactly. At the same time, for any given consistent initial values $(p_{0}, q_0)$ and small step size $h>0$, it is proved that there exists $\alpha^*=\alpha(h, p_0, q_0)$ such that the Hamiltonian energy can also be exactly preserved at each step. 
We provide a new variational formulation for symplectic PRK schemes and use it to prove that the parametrized PRK methods can preserve the quadratic invariants for Hamiltonian systems subject to holonomic constraints.
The parametric $\alpha$-PRK methods are shown to have the same convergence rate as the usual PRK methods and perform very well in various numerical experiments.
\end{abstract}

\section{Introduction}\label{sec:Introd}
For a given differential equation, a numerical method is called a geometric numerical integrator if it can accurately preserve some of the geometric characteristics of its solution  \cite{hairer2006geometric}. The structure-preserving algorithm of differential equations has made a lot of important development, and the basic idea of maintaining the important structure of original differential equation for numerical methods is widely accepted.  For the Hamiltonian system, the symplectic integrator for preserving the symplectic geometry property of the solution has been proved to have very good orbital tracking ability for a long time and various excellent properties \cite{Feng1986Difference, Feng2003,hairer2006geometric, marsden2001discrete}. 
Many effective methods for constructing symplectic integrators, such as the methods based on variational integrators \cite{marsden2001discrete}, generating functions \cite{Feng2003}  and Rung-Kutta (RK) methods \cite{Sun1993,Sun1995}, have been developed and investigated.  With the establishment of backward error analysis \cite{BenettinOn,TangFormal} and discrete KAM theory \cite{ShangKAM}, the symplectic integrators for Hamiltonian system is becoming more and more perfect.
Many new related numerical methods such as  multi-symplectic methods for Hamiltonian partial differential equations \cite{Qin2011} and stochastic symplectic methods for random Hamiltonian systems \cite{ChenModified} have been well developed.  In addition to the symplectic property, another extremely important feature of the Hamiltonian system is the conservation of energy. Therefore, it is very important to keep the system energy by numerical method. The important energy preserving methods include discrete gradient method \cite{Gonzalez1996Time,MclachlanGeometric}, average vector field method \cite{QuispelA}, HBVMs \cite{Brugnano2009Hamiltonian} and spectral methods \cite{Zhang19}

When multiple important structures or physical quantities exist for a system, such as the symplectic structure and the energy for Hamiltonian system, a natural question arises:  is it possible to construct a numerical method that preserves several of them? Regarding the symplectic structure and the energy for Hamiltonian system,  one has unfortunately a negative answer in general for constant step size. In fact, it is proved \cite{GeLie} that for non-ingregrable systems if the method is symplectic and can conserve the Hamiltonian energy exactly, then it is the time advance map for the exact Hamiltonian systems up to a reparametrization of time. A similar negative results is proved in \cite{ChartierAn} for general Hamiltonian system by B-series method. But the above negative results do not prevent people from constructing numerical methods to maintain both the energy and symplectic structure of the Hamilton system in some weaker sense. 

A breakthrough work in this regard is the parameterized Gauss collocation method firstly developed by Brugnano in \cite{BrugnanoEnergy}.  
Consider the canonical Hamiltonian systems in the form 
\begin{equation} \label{eq:Hamilton}
\begin{split}
\begin{cases}
&\dot{y}=J\nabla H (y), \\
&y(t_0)=y_0\in\mathbb{R}^{2d},\\
\end{cases}
\quad J=
\begin{pmatrix}
0 & I \\
-I & 0
\end{pmatrix}
\in\mathbb{R}^{2d\times 2d},
\end{split}
 \end{equation}
where $y=(p^T, q^T)^{T}$ and $I$ is the identity matrix and $H$ is the Hamiltonian energy.  Brugnano et.al. introduced a nice idea to
develop a new family of Gauss type methods which share both the symplecticity-like and energy conservation features under suitable
conditions. More precisely, they define a family of RK methods $y_{1}(\alpha)=\Phi_{h}(y_{0},\alpha)$, where $h$ is the step size
of integration, $\alpha$ is a real parameter. This method satisfies the following three conditions simultaneously:
(i) for $\alpha=0$ one gets the Gauss collocation method of order
$2s$, $s$ is the number of stages of the RK method; (ii) for any fixed choice of $\alpha\neq 0$, the corresponding
method is of order $2s-2$ and satisfies the conditions $b_{i}a_{i,j}+b_{j}a_{j,i}=b_{i}b_{j}$, thus being a quadratic invariant-preserving RK method;
 (iii) for any choice of $y_{0}$ and in a neighborhood of $h$, there
exists a value of parameter $\alpha^{\ast}=\alpha^{\ast}(y_{0},h)$ such that $H(y_{1})=H(y_{0})$(energy conservation). The
resulting method $y_{1}=\Phi_{h}(y_{0},\alpha^{*})$ has order $2s$, preserves the energy and quadratic invariants
\footnote{For stand RK method $(c,A,b)$ with constant step size $h>0$, the sufficient condition of symplecticity is given by $b_{i}a_{i,j}+b_{j}a_{j,i}=b_{i}b_{j}$ for all $i, j=1,2,...,s$. For irreducible RK method, which can be intuitively understood to mean that this RK method is not equivalent to a RK method with a lower series, see more details in \cite[pp.187]{Hairer2006}, this condition is also necessary, and also sufficient and necessary for quadratic invariants-preserving.  However, for the parametric $\alpha$-RK methods $y_{1}(\alpha)=\Phi_{h}(y_{0},\alpha^{\ast})$ that preserve the energy,  the parameter depends on the initial values, i.e., $\alpha^{\ast}=\alpha^{\ast}(y_{0},h)$. Then, $\alpha$-RK methods with condition $b_{i}a_{i,j}+b_{j}a_{j,i}=b_{i}b_{j}$ are in general not symplectic from the definition, but can preserve all the quadratic invariants.}.

This method  has been rewritten in the framework of discrete linear integral methods \cite{BrugnanoAnalysis}, which leads to a more refined theoretical analysis and a nice practical implementation strategy for seeking the parameter $\alpha^{*}$. Another extension form $\alpha$-RK to $\alpha$-PRK to make use of the additive structure of Hamiltonian system is given in \cite{WangParametric}. 
The key of constructing parameterized RK or PRK methods in \cite{BrugnanoEnergy,WangParametric} lies in the so-called $W$-transform technique. 
However, this technique is in general not suitable for constrained Hamiltonian systems since it will broke the constrained manifolds. See more details in 
Section \ref{sec:EQIP}.

In this paper, we focus on the geometric integrators for Hamiltonian systems subject to holonomic constraints. 
Consider the Hamiltonian with holonomic constraints (constraints that depend on $q$ only)  \cite[Chap. VII]{hairer2006geometric}
\begin{gather}\label{eq:constraintHamil}
\widetilde{H}(p, q, \lambda)=H(p, q)+g(q)^T\lambda,
\end{gather}
where $H: \R^d\times\R^d\to \mathbb{R}$ is the Hamiltonian function without constraint and $g(\cdot): \R^d\to \R^m$ is the constraint function.
Here, $\lambda\in \mathbb{R}^m$ is the Lagrangian multiplier.
The ODE corresponding to \eqref{eq:constraintHamil} is given by
(see \cite[VII.1.2]{hairer2006geometric}) 
\begin{equation} \label{eq:conHam}
\begin{split}
&\dot{q}=\nabla_pH(p, q)=:H_p(p, q),\\
&\dot{p}=-\nabla_qH(p, q)- G^{T}(q)\lambda=:-H_q(p, q)- G^{T}(q)\lambda,\\
&0=g(q),
\end{split}
\end{equation}
where we introduced
$G(q)=\left(\frac{\partial g_i}{\partial q_j}\right)_{m\times d}$
to be the Jacobian matrix of $g(q)$ so that
$G^{T}(q)=\nabla_qg \in \mathbb{R}^{d\times m}.$
Note that we use the convention $(\nabla_q g)_{ij}=\frac{\partial g_j}{\partial q_i}$ which is commonly used in the community of fluid mechanics.
Throughout this paper, we assume that the matrix $G(q)$ has full rank and $G(q)H_{pp}(p, q)G(q)^{T}$ is invertible, which will allow us to express $\lambda$ in terms of $(p, q)$. The dynamics given by \eqref{eq:conHam}
is on the manifold
\begin{equation} \label{eq:manifold}
\begin{split}
\mathcal{M}=\left\{(p,q);\quad  g(q)=0,~ G(q)H_p(p, q)=0 \right\},
\end{split}
\end{equation}
which is the cotangent bundle of the manifold given by 
$\mathcal{Q}:=\{q: g(q)=0\}. $
See \cite[VII.1.2]{hairer2006geometric} for more details.
The constraint $G(q)H_p(p, q)=0$ is also called the hidden constraint. 

It is easily to verify that the ODE \eqref{eq:conHam} has many geometric structures. First of all, the Hamiltonian (or energy) $H$ is preserved:
\begin{gather}
\frac{d}{dt}H(p(t), q(t))=0.
\end{gather}
Secondly, the symplecticity is also preserved. Consider a map
$\varphi: \cM\to \cM$ and its tangent map $\varphi': T_x\cM\to T_x\cM$. 
One says the map $\varphi$ is symplectic if
\begin{gather}
(\varphi'(x)\xi_1)^T J(\varphi'(x)\xi_2)=\xi_1^TJ\xi_2,~\forall \xi_1,\xi_2\in T_x\mathcal{M}.
\end{gather}
For $\cM$ being a submanifold of $\R^{2d}$, $\varphi'=(\nabla_x\varphi)^T$ so that
\[
\varphi'\xi=\xi\cdot\nabla_x\varphi=\left.\frac{d}{d\tau}\right|_{\tau=0}\varphi(\gamma(\tau))
\]
for any smooth curve $\gamma$ in $\cM$ satisfying $\gamma(0)=x, \gamma'(0)=\xi$. 
It can be shown that the flow map of \eqref{eq:conHam} is symplectic.
Plainly speaking, the symplecticity means that the flow preserves the quadratic forms, which further implies that the volume form is conserved
(see the introduction of \cite{jay1996symplectic}).

From the pointview of structure-preserving methods, if one solves the Hamiltonian system \eqref{eq:conHam} numerically, besides the basic requirement that the solution falls onto the manifold $\cM$, one also tends to ask that the energy preservation or the symplecticity is satisfied.
This additional requirement for staying on manifolds $\mathcal{M}$ for constrained Hamiltonian systems presents new difficulties in the construction of structure preserving numerical schemes.  

There are many symplectic methods for the holonomic constrained Hamiltonian system. One typical class is the PRK methods as detailed explained in \cite{jay1996symplectic,marsden2001discrete,jay2002iterative}. The first order method in 
\cite[sec VII.1.3]{hairer2006geometric} and the Shake-Rattle method in \cite{andersen1983rattle, hairer2006geometric} are special cases of the PRK methods.
Symplectic variational integrators for constrained Hamiltonian systems are developed in \cite{marsden2001discrete,Wenger2017Construction}.
Those methods are symplectic and preserve the manifolds $\mathcal{M}$ exactly, but they are not Hamiltonian energy converved. The second order energy-conserving method in the line integral framework developed in \cite{BrugnanoLine} is not symplectic in general and can not conserve the the hidden constraint $G(q)H_p(p, q)=0$.  The projected RK methods are proposed in \cite{WeiProjected}, but it is not symplectic.

The goal of this work is to introduce a parameter in the PRK methods so that both the energy and the quadratic invariants can be preserved for the constrained Hamiltonian system following the spirit of \cite{BrugnanoEnergy,WangParametric}.
Note that if we do this method for every trajectory, then the energy can be preserved for all trajectories and the method is in general not symplectic.
However,  if we fix down the parameters, the method is symplectic and can be preserve energy for one trajectory.

The rest of the paper is arranged as follows. In section \ref{sec:symPRK}, we recall some of the basic concepts and properties of symplectic PRK methods for constrained Hamiltonian systems. In order to show the symplectic PRK methods preserve the quadratic invariants for Hamiltonian systems subject to holonomic constraints, we provide a new variational formulation for symplectic PRK methods in section \ref{sec:variationalfor}.  In section \ref{sec:EQIP}, we construct parametrization $\alpha$-PRK methods based on Shake-Rattle method and Lobatto IIIA-IIIB pairs, and check various properties of the new schemes.  Numerical examples are included in section \ref{sec:numer} to illustrate the energy and quadratic invariants conservations of the $\alpha$-PRK methods.

\section{Symplectic Partitioned Runge-Kutta methods}
\label{sec:symPRK}

In this subsection, we give a brief review of the symplectic PRK methods. See \cite{jay1996symplectic,hairer2006geometric} for more details.
The $s$-stage PRK methods denoted by $\left(c, A, b; \widehat{A},  \widehat{b} \right)$ for the Hamiltonian ODE \eqref{eq:conHam} is given by the following
\begin{gather}\label{eq:PRK}
\begin{split}
& p_1=p_0+h\sum_{i=1}^s \hat{b}_i \ell_i,~~~
q_1=q_0+h\sum_{i=1}^s b_i k_i.\\
& \ell_i=-H_q(P_i, Q_i)- G^{T}(Q_i) \Lambda_i,~
k_i=H_p(P_i, Q_i), ~g(Q_i)=0,\\
& P_i=p_0+h\sum_{j=1}^s \hat{a}_{ij}\ell_j,~~~
Q_i=q_0+h\sum_{j=1}^s a_{ij}k_j.
\end{split}
\end{gather}
The equations considered here are autonomous, so the parameters $c$ for the Runge-Kutta methods are absent in \eqref{eq:PRK}. See section \ref{sec:PRKIII} for more details.

\begin{lemma}
\label{lem:symcondition}
\cite{jay1996symplectic} For a given a PRK method, if the conditions for symplecticity 
\begin{gather}\label{eq:prksymplecticcond}
\begin{split}
& b_i=\hat{b}_i,\\
& b_i \hat{a}_{ij}+\hat{b}_ja_{ji}=b_i \hat{b}_j,
\end{split}
\end{gather}
 is satisfied, then the method for Hamiltonian systems with holonomic constraints is also symplectic.
\end{lemma}

For $s$-stage PRK methods, if $a_{ij}, \hat{a}_{ij}, b_i, \hat{b}_i$ are given, then there are $3s$ 
equations for the $3s$ unknowns ($P_i, Q_i, \Lambda_i$).
However, for general symplectic PRK methods, $(p_1, q_1)$ may not fall onto $\cM$.  To ensure that $q_1$ satisfies the constraint, it is suggested in \cite{jay1996symplectic} to impose the conditions
\begin{gather}\label{eq:prkstiffaccurate}
a_{1j}=0,~~~a_{sj}=b_j,~~1\le j\le s,
\end{gather}
so that
\begin{gather}
Q_1\equiv q_0,~~~Q_s\equiv q_1.
\end{gather}
With this, one has correspondingly
\begin{gather}
\hat{a}_{js}=0,~\hat{a}_{j1}=b_1, ~~ j=1,2,...,s,
\end{gather}
from the symplectic condition given in Lemma \ref{lem:symcondition}. This requirement removes the unknowns $Q_1,\Lambda_s$ and the equations  $g(Q_1)=0$, $Q_1=q_0+h\sum_{j=1}^s a_{ij}k_j$.
 Hence,  the number of equations is now equal to the number of unknowns 
 $\left(Q_2,\ldots, Q_s, P_1,\ldots, P_s, \Lambda_1,\ldots, \Lambda_{s-1}\right)$.
In fact, it is shown in \cite{jay1996symplectic} that these variables can be uniquely solved when $h$ is small enough under some reasonable assumptions on $H$ and $\left(c, A, b; \widehat{A},  \widehat{b} \right)$.
We remark that even though one imposes $Q_1=q_0$, the Lagrange multiplier $\Lambda_1$ is still in the system of equations.
Lastly, $p_1$ and $\Lambda_s$ can be determined by the hidden constraint condition
$0=G(q_1)\cdot H_p(p_1, q_1).$

For the PRK methods suggested in \cite{jay1996symplectic}, we hence solve for the next two nonlinear systems in turn at each step.
The first one 
\begin{equation} \label{eq:PRK1}
\begin{split}
\begin{cases}
&Q_i=q_0+h\sum\limits_{j=1}^s a_{ij}H_p(P_j, Q_j),\\
&P_i=p_0-h\sum\limits_{j=1}^{s-1} \hat{a}_{ij} \left(H_q(P_j, Q_j)- G^{T}(Q_j) \Lambda_j \right),\\
&0=g(Q_i).\\
\end{cases}
\end{split}
 \end{equation}
 with $Q_1\equiv q_0, Q_s\equiv q_1.$
Then, we can solve $(p_1, \Lambda_s)$ by
 \begin{equation} \label{eq:PRK2}
\begin{cases}
p_1&=p_0-h\sum\limits_{i=1}^{s-1} \hat{b}_{i} \left(H_q(P_i, Q_i)- G^{T}(Q_i) \Lambda_i \right)-h\hat{b}_{s} \left(H_q(P_s, Q_s)- G^{T}(Q_s) \Lambda_s \right),\\
0&=G(q_1)H_{p}(p_1, q_1).
\end{cases}
 \end{equation}

As commented in \cite[sec VI.6.2]{hairer2006geometric}, the symplectic 
schemes can be written as certain variational integrators. In \cite[sec VI.6.3]{hairer2006geometric}, it was explained in detail how the PRK schemes for unconstrained problems can be formulated into a variational integrator. The same is true for Hamiltonian systems with holonomic constraints. In fact, Marsden and West addressed this issue in 
\cite[section 3.5.6]{marsden2001discrete}.
In particular, given $(q_0, q_1)\in \mathbb{R}^d\times \mathbb{R}^d$, one can implicitly define quantities 
$\left(\bar{p}_0, \bar{p}_1, \bar{Q}_i, \bar{P}_i, \dot{\bar{P}}_i, \dot{\bar{Q}}_i \right)$
for $i=1,\ldots, s$ and $\bar{\Lambda}_i$ for $i=2,\ldots, s-1$ through the following system of equations
\begin{gather}\label{eq:mvconstraints}
\begin{split}
q_1&=q_0+h\sum_{j=1}^s b_j\dot{\bar{Q}}_j, \quad \bar{p}_1=\bar{p}_0+h\sum_{j=1}^s \hat{b}_j\dot{\bar{P}}_j, \\
\bar{Q}_i&=q_0+h\sum_{j=1}^s a_{ij}\dot{\bar{Q}}_j, \quad \bar{P}_i=\bar{p}_0+h\sum_{j=1}^s \hat{a}_{ij}\dot{\bar{P}}_j,~~i=1,\ldots, s,\\
\dot{\bar{Q}}_i&=H_p (\bar{Q}_i,\bar{P}_i),~~i=1,\ldots, s;   \dot{\bar{P}}_1
=-H_q(\bar{Q}_1, \bar{P}_1), ~~\dot{\bar{P}}_s
=-H_q(\bar{Q}_s, \bar{P}_s)\\
\dot{\bar{P}}_i&= -H_q(P_i, Q_i)- G^{T}(Q_i) \Lambda_i,~~~
0=g(Q_i),~~i=2,\ldots, s-1.
\end{split}
\end{gather}
Note that the definitions of $\dot{\bar{P}}_1$ and $\dot{\bar{P}}_s$ above
are different from other $\dot{\bar{P}}_i$ (in fact, $\dot{\bar{P}}_i=\dot{P}_i$ for $i\neq 1,s$ while there should be some corrections for $\dot{\bar{P}}_1$ and $\dot{\bar{P}}_s$ to define $\dot{P}_1, \dot{P}_s$). Moreover, $\bar{p}_1$ and $\bar{p}_0$ are not $p_1, p_0$ in \eqref{eq:PRK}. They are different by some terms involving the corresponding Lagrange multipliers.  See the proof of \cite[Theorem 3.5.1]{marsden2001discrete}
for the details. With the quantities defined in \eqref{eq:mvconstraints}, one can define the discrete Lagrangian
\begin{gather}\label{eq:mvdiscretelag}
L_h(q_0, q_1)=h\sum_{i=1}^s b_i L(\bar{Q}_i, \dot{\bar{Q}}_i),
\end{gather}
where $L(q, \dot{q})$
is the Lagrangian for the time-continuous dynamics.
The Euler-Lagrange equation for this discrete Lagrangian under the holonomic constraints
$(q_0, q_1)\in \mathcal{Q}\times \mathcal{Q}$
reads
\begin{gather}\label{eq:discreteELconstraint}
\begin{split}
& p_0=-\frac{\partial L_h}{\partial x}(q_0, q_1)+h G^T(q_0)\cdot \lambda_0,\\
& p_1=\frac{\partial L_h}{\partial y}(q_0, q_1)-h G^T(q_1)\cdot \lambda_1,\\
& g(q_1)=0,\\
& G(q_1)H_p(p_1, q_1)=0.
\end{split}
\end{gather}
See \cite[eqs. (3.5.2a-d)]{marsden2001discrete}. Note that though the Euler-Lagrange equations are derived by fixing $q_0, q_1$, one should regard $p_0, q_0$ as knowns  and solve $(p_1, q_1)$ dynamically.
The good thing of this variational formulation is that many conservation properties can be derived via the discrete Noether theorem. See \cite{marsden2001discrete} and also section \ref{sec:variationalfor} below for more details on conservation properties and discrete Noether theorem. However, for quadratic invariants, the form  \eqref{eq:mvconstraints}-\eqref{eq:mvdiscretelag} is not convenient to use and we will propose another variational formulation below in section \ref{sec:variationalfor}.

\subsection{PRK schemes: Shake-Rattle and Lobatto IIIA-IIIB pairs}
\label{subsec:shakerattle}

We first recall some famous PRK methods, including Shake-Rattle algorithm and Lobatto IIIA-IIIB pairs, 
which will be used to construct the parameterization $\alpha$-PRK later.
For typical separable Hamiltonian 
\[
H(p, q)=\frac{1}{2}p^TM^{-1}p+U(q),
\]
the traditional Shake method reads
\[
\begin{split}
& q_{n+1}-2q_n+q_{n-1}=-h^2M^{-1}(U_q(q_n)+G(q_n)^T\Lambda_n ),\\
& 0=g(q_{n+1})
\end{split}
\]
Then, the momentum is
$p_n=M(q_{n+1}-q_{n-1})/(2h).$
The corresponding Hamiltonian formulation for the Shake method is
\[
\begin{split}
& p_{n+1/2}=p_n-\frac{h}{2}(U_q(q_n)+G(q_n)^T\Lambda_n ),\\
& q_{n+1}=q_n + h M^{-1}p_{n+1/2},~~g(q_{n+1})=0,\\
& p_{n+1}=p_{n+1/2}-\frac{h}{2}(U_q(q_{n+1}) +G(q_{n+1})^T\Lambda_{n+1})
\end{split}
\]
Unfortunately, so-defined $p_n$ may not satisfy the constraint. 
Anderson \cite{andersen1983rattle} proposed the Rattle algorithm. That is to determine $p_{n+1}$ by using another multiplier $\mu_n$ which can be different from $\Lambda_{n+1}$:
\[
p_{n+1}=p_{n+1/2}-\frac{h}{2}(U_q(q_{n+1}) +G(q_{n+1})^T\mu_n ).
\]
Then, one uses the condition
\[
G(q_{n+1})\cdot M^{-1}p_{n+1}=0
\]
to determine $\mu_n$ and thus $p_{n+1}$. Clearly, this Shake-Rattle algorithm is one of the symplectic PRK method mentioned above. 
The more general form of Shake-Rattle for the general Hamiltonian system with  holonomic constraints reads
\begin{gather}
\begin{split}
& p_{n+1/2}=p_n-\frac{h}{2}(\nabla_qH(p_{n+1/2}, q_n)+G^T(q_n) \Lambda_n ),\\
& q_{n+1}=q_n + \frac{h}{2}(\nabla_pH(p_{n+1/2}, q_n)+\nabla_pH(p_{n+1/2}, q_{n+1})),~~~~g(q_{n+1})=0,\\
& p_{n+1}=p_{n+1/2}-\frac{h}{2}(\nabla_qH(p_{n+1/2}, q_{n+1}) +G^T(q_{n+1})\Lambda_{n+1})\\
& G(q_{n+1})\cdot p_{n+1}=0.
\end{split}
\end{gather}
It is proved that the Shake-Rattle algorithm is symmetric, symplectic and convergent of order two \cite[sec VII.I]{hairer2006geometric}.

The Shake-Rattle algorithm is second order. A nice extension to higher order is the Lobatto IIIA-IIIB pairs developed in \cite{jay1996symplectic}.
This method combined with the projection step by the hidden constraint condition defined in the manifold $\mathcal{M}$ has the following merits.
(a) preserving the numerical solutions on the manifold $\mathcal{M}$ exactly; (b) it is symmetric, symplectic and super convergent of order $2s-2$, the $s$ is the stage of PRK method. In particular, for $s=3$, the coefficients for Lobatto IIIA-IIIB pairs are given by 
\begin{equation} \label{eq:3IIIA-IIIB0}
\begin{split}
\begin{array}{c|ccc}
  0& 0 & 0 &0 \\
  \frac{1}{2}& \frac{5}{24} &\frac{1}{3} &-\frac{1}{24}\\
  1& \frac{1}{6} & \frac{2}{3} &  \frac{1}{6}\\
  \hline
  & \frac{1}{6} & \frac{2}{3} &  \frac{1}{6}\\
\end{array},\quad\quad~
\begin{array}{c|ccc}
  0& \frac{1}{6} & -\frac{1}{6} &0\\
  \frac{1}{2}& \frac{1}{6} &\frac{1}{3} &0\\
  1& \frac{1}{6} & \frac{5}{6} &  0\\
  \hline
  & \frac{1}{6} & \frac{2}{3} &  \frac{1}{6}\\
\end{array}.
\end{split}
 \end{equation}

\section{An alternative variational formulation for symplectic PRK schemes with holonomic constraints}
\label{sec:variationalfor}

In this section, we aim to study the quadratic invariants of the PRK schemes via the discrete Noether's theorem. For this purpose, we need to find some equivalent variational forms for the PRK scheme. However, the form in the work of Marsden and West (i.e. \eqref{eq:mvconstraints}-\eqref{eq:mvdiscretelag}) is not convenient, as it involves Lagrangian multipliers and there is nonlinearity. Instead, we will propose an alternative formulation so that the symmetry can be studied better.

It is known that symplectic PRK schemes for Hamiltonian systems without constraints can conserve quadratic invariants of the form
\begin{gather}
I(p, q)=q^TDp,
\end{gather}
where $D$ is some fixed matrix (\cite[sec VI.2.2]{hairer2006geometric}). 
Naturally, one is curious whether there is an analogue for Hamiltonian systems with holonomic constraints. Clearly, due to the constraint $g(q)=0$, the quadratic invaraints must be of some particular forms. Instead, motivated by Noether's theorem \cite{noether71,arnold13}, we will consider that 
\begin{gather}
\mathscr{D}:=\Big\{D: g\left(e^{sD}q\right)=g(q),~~L\left(e^{sD}q, e^{sD}\dot{q}\right)
=L(q,\dot{q}),~\text{for all}~s\in\R \Big\},
\end{gather}
Clearly, by Noether's theorem, we have the following.
\begin{lemma}
For any $D\in \mathscr{D}$, the quantity $q^TDp$ is a first integral of the Hamiltonian system with holonomic constraints.
\end{lemma}

We now aim to show that the symplectic PRK schemes will conserve all the quadratic invaraints of the form $q^TDp$ for $D\in \mathscr{D}$.
\begin{theorem}\label{thm:quadraticinvar}
The symplectic PRK schemes for Hamiltonian system with holonomic constraints defined in \eqref{eq:PRK} conserve all quadratic invariants of the form 
\begin{gather}
I(p, q)=q^TDp,~~D\in \mathscr{D}.
\end{gather}
\end{theorem}

To do this, we first recall the discrete Noether's theorem for variational integrators without constraints.
\begin{lemma} [Discrete Noether's theorem]\label{lmm:discretenoether}
Assume the discrete Lagrangian $L_h(q_0,q_1)$ is invariant under a one-parameter group of transformation $\{\mathfrak{g}_s: s\in \mathbb{R} \}$: 
$L_h\left(\mathfrak{g}_s(q_0), \mathfrak{g}_s(q_1)\right)=L_h(q_0, q_1)$ for all $s\in\R$ and $(q_0, q_1)$. Then the corresponding variational integrators for Lagrangian systems have the first integral in the form of 
\[
p_{n+1}^T a(q_{n+1})=p_{n}^T a(q_n)
\]
 where $a(q)=\frac{d}{ds}\mathfrak{g}_s(q)|_{s=0}$ is the generator for the group.
Furthermore, when there are holonomic constraints, besides the conditions above, if moreover $\{\mathfrak{g}_s: s\in \mathbb{R} \}$ leaves the constraints manifold $\mathcal{Q}$ invariant, then the claim still holds.
\end{lemma}
The first part can be found in \cite[sec VI.6]{hairer2006geometric}. The second part has been remarked in \cite[sec. 3.4.2]{marsden2001discrete}. For the convenience, we provide a direct verification here.
\begin{proof}[Proof of the second part of Lemma \ref{lmm:discretenoether}]
Taking derivative on $s$ and setting $s=0$ in $L_h\left(\mathfrak{g}_s(q_0), \mathfrak{g}_s(q_1)\right)=L_h(q_0, q_1)$, one has
\[
a(q_0)\cdot \frac{\partial L_h}{\partial x}(q_0, q_1)+a(q_1)\cdot \frac{\partial L_h}{\partial y}(q_0, q_1)=0.
\]
Since $\{\mathfrak{g}_s: s\in \mathbb{R} \}$ leaves the constraints manifold invaraint, 
\[
G(q)a(q)=0 \Rightarrow a(q)\cdot G^T(q)\lambda=0,~\text{for all~}\lambda.
\]
Using \eqref{eq:discreteELconstraint}, one easily finds $p_{1}^T a(q_{1})=p_{0}^T a(q_0)$.
\end{proof}

Below, we propose an alternative variational formulation for the PRK scheme that appears different from  \eqref{eq:mvconstraints}-\eqref{eq:mvdiscretelag}. This formulation is suitable to verify
the conditions for the discrete Noether's theorem.
\begin{proposition}\label{pro:variationalPRK}
The discrete Lagragian for $(q_0, q_1)\in \R^d\times \R^d$ for the symplectic PRK methods (\eqref{eq:PRK} with conditions \eqref{eq:prksymplecticcond}-\eqref{eq:prkstiffaccurate}) can be defined as
\begin{gather}\label{eq:discreteLg}
L_h(q_0, q_1)=\mathrm{ext}_{\left\{\dot{Q}_i \right\}}h \sum_{i=1}^s b_i L(Q_i, \dot{Q}_i)
\end{gather}
with constraints
\begin{eqnarray}
& g(Q_i)=0,~~i=2,\ldots, s-1, \label{eq:qconstraint}\\
& q_1=q_0+h\sum_{j=1}^s a_{ij}\dot{Q}_j \label{eq:q0q1constraint}.
\end{eqnarray}
Above, ``$\mathrm{ext}$'' means extremizing and 
$Q_i=q_0+h\sum_{j=1}^s a_{ij}\dot{Q}_j$.
\end{proposition}

\begin{remark}
Note that we are not defining $Q_1, Q_s$ as $a_{si}=b_i,~~a_{1i}=0, b_i\neq 0$, so that
$Q_1=q_0$, $Q_s=q_1$. In this definition, $q_0, q_1$ do not have to be in $\mathcal{Q}$.
\end{remark}

\begin{proof}[Proof of Proposition \ref{pro:variationalPRK}]

{\bf Step--1:}
First of all, we note that the discrete Lagrangian is well-defined.
The conditions for the extremizers are given by
\begin{gather}\label{eq:eqnforext}
\begin{split}
&\left(\sum_{i=1}^s b_i \frac{\partial L}{\partial q}(Q_i, \dot{Q}_i) ha_{ij}\right)
+b_j\frac{\partial L}{\partial \dot{q}}(Q_j, \dot{Q}_j)
=\sum_{i=2}^{s-1}b_i ha_{ij} G^T(Q_i)\Lambda_i +b_j\lambda,\\
& g(Q_i)=0,~i=2,\ldots, s-1,\\
& q_1=q_0+h\sum_{j=1}^s a_{ij}\dot{Q}_j.
\end{split}
\end{gather}
Note that we have built in the constraint $g(Q_i)=0$ by using the multiplier
$b_i\Lambda_i$.
Since $b_i\neq 0$, this will not change anything.
These $sd$ equations together with $g(Q_i)=0$ ($s-2$ equations)
and the constraint ($d$ equations) can yield the solutions $(\{\dot{Q}_j\}_{j=1}^s, \{\Lambda_i\}_{i=2}^{s-1}, \lambda)$ for $h$ sufficiently small. 


Here we are solving $\dot{Q}_i$ in terms of $q_0, q_1$ for any $(q_0, q_1)$ in $\mathbb{R}^d\times \mathbb{R}^d$, not just for $(q_0, q_1)\in \mathcal{Q}\times \mathcal{Q}$.
With this observation, we can take differentiation on $q_{\ell},\ell=0,1$ for the constraint \eqref{eq:q0q1constraint} to find
\begin{gather}\label{eq:constraintq1var}
h\sum_{j=1}^s b_j \frac{\partial \dot{Q}_j}{\partial q_0}=-I,~~~
h\sum_{j=1}^s b_j \frac{\partial \dot{Q}_j}{\partial q_1}=I,
\end{gather}
and for the constraints \eqref{eq:qconstraint} to find
\begin{gather}\label{eq:constraintQvar}
h\left(\sum_{j=1}^s a_{ij}\frac{\partial \dot{Q}_j}{\partial q_0}\right)G^T(Q_i)=-G^T(Q_i),
~~~h\left(\sum_{j=1}^s a_{ij}\frac{\partial \dot{Q}_j}{\partial q_1}\right)G^T(Q_i)=0.
\end{gather}
Note that here we used the convention 
\[
\frac{\partial \dot{Q}_j}{\partial q_0}:=\nabla_{q_0}\dot{Q}_j,
~~~\left(\frac{\partial \dot{Q}_j}{\partial q_0}\right)_{\ell,m}=\frac{\partial \dot{Q}_j^{(m)}}{\partial q_0^{(\ell)}}.
\]
The reason to do this is that we have made $\frac{\partial L}{\partial q}$
a column vector.

{\bf Step--2:} We derive the equation for $p_i$ defined in \eqref{eq:discreteELconstraint}.
By the first equation in \eqref{eq:discreteELconstraint}, one has
\begin{gather*}
\begin{split}
p_0
&=-\left(h\sum_{i=1}^s b_i \left[I+h\sum_{j=1}^s a_{ij}\frac{\partial \dot{Q}_j}{\partial q_0}\right] \frac{\partial L}{\partial q}(Q_i,\dot{Q}_i)
+h\sum_{i=1}^s b_i  \frac{\partial\dot{Q}_i}{\partial q_0}\frac{\partial L}{\partial \dot{q}}(Q_i,\dot{Q}_i)\right)+h G^T(q_0)\cdot \lambda_0\\
&=-h\sum_{i=1}^s b_i \frac{\partial L}{\partial q}(Q_i,\dot{Q}_i)
-h\sum_{j=1}^s \frac{\partial \dot{Q}_j}{\partial q_0} \left[\sum_{i=2}^{s-1} h b_i a_{ij}G^T(Q_i)\Lambda_i+b_j\lambda \right]
+h G^T(q_0)\cdot \lambda_0\\
&=-h\sum_{i=1}^s b_i \frac{\partial L}{\partial q}(Q_i,\dot{Q}_i)
+h\sum_{i=2}^{s-1}b_i G^T(Q_i)\Lambda_i+\lambda+h G^T(q_0)\cdot \lambda_0.
\end{split}
\end{gather*}
Above, the second inequality follows from \eqref{eq:eqnforext} while last inequality follows from \eqref{eq:constraintq1var} and \eqref{eq:constraintQvar}.
Similarly, we can compute by the second equation in \eqref{eq:discreteELconstraint} that
\begin{gather*}
\begin{split}
p_1 &=h\sum_{i=1}^s b_i h\sum_{j=1}^s a_{ij}\frac{\partial \dot{Q}_j}{\partial q_1}\frac{\partial L}{\partial q}(Q_i,\dot{Q}_i)
+h\sum_{i=1}^s b_i  \frac{\partial\dot{Q}_i}{\partial q_1}\frac{\partial L}{\partial \dot{q}}(Q_i,\dot{Q}_i)-hG^T(q_1)\cdot\lambda_1\\
&=h\sum_{j=1}^s\frac{\partial \dot{Q}_j}{\partial q_1}
\left[ \sum_{i=2}^{s-1} h b_i a_{ij}G^T(Q_i)\Lambda_i+b_j\lambda\right]
-hG^T(q_1)\cdot\lambda_1\\
&=\lambda-hG^T(q_1)\cdot\lambda_1.
\end{split}
\end{gather*}
Defining that
$\Lambda_1:=\lambda_0/b_1,  \Lambda_s:=\lambda_1/b_s,$
one has
\begin{gather}
p_1=p_0+h\sum_{i=1}^s b_i \frac{\partial L}{\partial q}(Q_i,\dot{Q}_i)
-h\sum_{i=1}^{s}b_i G^T(Q_i)\Lambda_i.
\end{gather}

{\bf Step--3:}
We verify that the Euler-Lagrange equations \eqref{eq:discreteELconstraint} given by so-defined
discrete Lagrange is equivalent to the symplectic PRK methods with constraints.
Recall that the (continuous) Hamiltonian is related to Lagrangian by
$H(p, q)=p\cdot \dot{q}-L(q, \dot{q}),~~p=\frac{\partial L}{\partial \dot{q}},$
where $\dot{q}=\dot{q}(p, q)$ is determined by the second equation.
It can be computed that
$\dot{q}=\frac{\partial H(p,q)}{\partial p},~~~\frac{\partial H(p,q)}{\partial q}
=-\frac{\partial L(q,\dot{q})}{\partial q}.$
Hence, we have
\begin{eqnarray}
Q_i=q_0+h\sum_{j=1}^s a_{ij}H_p(P_i, Q_i),
~~~q_1=q_0+h\sum_{j=1}^s a_{ij}H_p(P_i, Q_i),\label{eq:conclude1}\\
p_1=p_0-h\sum_{i=1}^s b_i H_q(P_i, Q_i)
-h\sum_{i=1}^{s}b_i G^T(Q_i)\Lambda_i,\label{eq:p0p1eqn}
\end{eqnarray}
where
$P_i=\frac{\partial L(Q_i, \dot{Q}_i)}{\partial \dot{Q}_i}.$
The first equation in \eqref{eq:eqnforext} can then be rewritten as
\[
\sum_{i=1}^s hb_ia_{ij}(-H_q(P_i, Q_i))
+b_j P_j=\sum_{i=2}^{s-1}b_i h a_{ij}G^T(Q_i)\Lambda_i
+b_j[p_1+hb_sG^T(q_1)\cdot\Lambda_s]
\]
Using \eqref{eq:p0p1eqn}, we then have
\begin{gather}
P_j=p_0-h\sum_{i=1}^s[b_i-b_ia_{ij}/b_j]H_q(P_i, Q_i)
-\sum_{i=2}^{s-1}[b_i-b_ia_{ij}/b_j] G^T(Q_i)\Lambda_i
-hb_1G^T(Q_1)\cdot\Lambda_1.
\end{gather}
Using the condition $a_{1j}=0, a_{sj}=b_j$, this is further simplified to
\begin{gather}\label{eq:concludePj}
P_j=p_0-h\sum_{i=1}^s[b_i-b_ia_{ij}/b_j]H_q(P_i, Q_i)
-\sum_{i=1}^{s}[b_i-b_ia_{ij}/b_j] G^T(Q_i)\Lambda_i.
\end{gather}
Defining 
$\hat{a}_{ji}=b_i-b_ia_{ij}/b_j,$
one has the desired form for $P_j$.

Eventually, we find \eqref{eq:conclude1}, \eqref{eq:p0p1eqn}, 
\eqref{eq:concludePj} and \eqref{eq:qconstraint}--\eqref{eq:q0q1constraint} then form a complete system that is the same as the symplectic PRK method.
\end{proof}

Lastly, we verify Theorem \ref{thm:quadraticinvar}.
\begin{proof}[Proof of Theorem \ref{thm:quadraticinvar}]
The constraint condition is obvious.
We verify that 
\begin{gather}\label{eq:discreteaux1}
L_h(e^{sD}q_0, e^{sD}q_1)=L_h(q_0, q_1).
\end{gather}
This is in fact a direct corollary of Proposition  \ref{pro:variationalPRK}.
Since $e^{sD}$ is a linear group, then for any $(q_0, q_1)$ and any
$ \{\dot{Q}_j\}_{j=1}^s$, we have a corresponding sequence of data $\{\dot{\bar{Q}}_j\}_{j=1}^s:= \{e^{sD}\dot{Q}_j\}_{j=1}^s$ for $(e^{sD}q_0, e^{sD}q_1)$.  Correspondingly the new $Q_i$ (denoted by $\bar{Q}_i$) is given by
\[
\bar{Q}_i=e^{sD}q_0+h\sum_{j=1}^s a_{ij}e^{sD}\dot{Q}_j
=e^{sD}Q_i,
\]
by the linearity of $e^{sD}: \mathbb{R}^d\to \mathbb{R}^d$. Since for the continuous Lagrangian one has
$L(e^{sD}Q_i, e^{sD}\dot{Q}_i)=L(Q_i, \dot{Q}_i)$,
we find the sums to extremize in \eqref{eq:discreteLg} have the same value at the corresponding data. Hence, their extrema must be the same and thus \eqref{eq:discreteaux1} follows. 
Applying the discrete Noether's theorem (Lemma \ref{lmm:discretenoether}), we obtain the desired result.
\end{proof}

\section{Energy and quadratic invariants preserving methods}
\label{sec:EQIP}

In this section, we propose a method that preserve the energy and quadratic invariants
for Hamiltonian system with constraints based on the symplectic PRK schemes, following the spirit of \cite{BrugnanoEnergy,WangParametric}.
Recall that the following conditions guarantee that the $s$-stage PRK method is symplectic and preserves the quadratic invariants:
\begin{gather}\label{eq:conditiongroup1}
\begin{split}
&\text{The symplecticity and quadratic invariants:~}  b_i=\hat{b}_i,  b_i \hat{a}_{ij}+\hat{b}_ja_{ji}=b_i \hat{b}_j, i, j=1,...,s.\\
&\text{For the constraints:~} a_{1j}=0,~a_{sj}=b_j, \text{or equivalently,~} \hat{a}_{is}=0,~~\hat{a}_{i1}=b_1,~i=1,\ldots, s.
\end{split}
\end{gather}

These schemes do not in general preserve the Hamiltonian energy. In fact, such methods conserve the symplectic properties and Hamiltonian energy for constant stepsize $h>0$ has been proven not to exist in general \cite{ChartierAn, GeLie}.
 The key observation now is that the conditions imposed above for preserving the symplecticity and quadratic invariants and constraints manifolds cannot determine $\left( c, A, b; \hat{A}, \hat{b}\right)$ totally, so there is actually some freedom to choose these parameters.

Inspired by \cite{BrugnanoEnergy,WangParametric}, we are motivated to construct $\alpha$-PRK methods 
so that the energy is also preserved. 
However, the specific construction technique developed in \cite{BrugnanoEnergy,WangParametric}  is not feasible here in general.
 In fact, the key in the construction of $\alpha$-PRK methods in the two paper is the so called $W$-transformation, which allows us to express the symplectic condition in a very compact form with the help of matrix $X_{G}$, where the three diagonal anti-symmetric matrix $X_{G}$ is the $W$-transformation matrix of Gauss collocation method. Thus, we can parameterize the transformed matrix so that it remains to satisfy the same conditions as usual symplectic RK or PRK methods, and then obtain the parameterized $\alpha$-PRK method by inverse $W$-transformation. 
For example, the $2$-stage and $3$-stage $\alpha$-PRK methods based on Lobatto IIIA-IIIB pairs are given by \cite{WangParametric}
\begin{equation} \label{eq:2IIIA-IIIB}
\begin{split}
\begin{array}{c|cc}
  0& \frac{1}{6}\alpha & -\frac{1}{6}\alpha\\
  1& \frac{1}{2}-\frac{1}{6}\alpha & \frac{1}{2}+\frac{1}{6}\alpha\\
  \hline
   & \frac{1}{2} & \frac{1}{2}\\
\end{array},\quad\quad~
\begin{array}{c|cc}
 0& \frac{1}{2}-\frac{1}{6}\alpha & \frac{1}{6}\alpha\\
  1& \frac{1}{2}+\frac{1}{6}\alpha & -\frac{1}{6}\alpha\\
  \hline
   & \frac{1}{2} & \frac{1}{2}\\
\end{array}.
\end{split}
 \end{equation}
 
\begin{equation} \label{eq:3IIIA-IIIB}
\begin{split}
\begin{array}{c|ccc}
  0& \frac{1}{5}\alpha & -\frac{2}{5}\alpha &\frac{1}{5}\alpha \\
  \frac{1}{2}& \frac{5}{24}-\frac{1}{15}\alpha &\frac{1}{3}+\frac{2}{15}\alpha &-\frac{1}{24}-\frac{1}{15}\alpha\\
  1& \frac{1}{6}+\frac{1}{5}\alpha & \frac{2}{3}-\frac{2}{5}\alpha &  \frac{1}{6}-\frac{1}{5}\alpha\\
  \hline
  & \frac{1}{6} & \frac{2}{3} &  \frac{1}{6}\\
\end{array},\quad\quad~
\begin{array}{c|ccc}
  0& \frac{1}{6}-\frac{1}{5}\alpha & -\frac{1}{6}+\frac{2}{5}\alpha &-\frac{1}{5}\alpha \\
  \frac{1}{2}& \frac{1}{6}+\frac{1}{15}\alpha &\frac{1}{3}-\frac{2}{15}\alpha &\frac{1}{15}\alpha\\
  1& \frac{1}{6}-\frac{1}{5}\alpha & \frac{5}{6}+\frac{2}{5}\alpha &  -\frac{1}{5}\alpha\\
  \hline
  & \frac{1}{6} & \frac{2}{3} &  \frac{1}{6}\\
\end{array}.
\end{split}
 \end{equation}
We can see that although the method keeps symplecticity and quadratic invariants, it generally does not satisfy the constraint conditions $a_{1j}=0, a_{sj}=b_j$, so its numerical solution cannot be guaranteed to fall on the manifold $\mathcal{M}$.

In this article, we introduce another new parameterization $\alpha$-PRK methods. According to \eqref{eq:conditiongroup1}, we focus only on $(A, b)$. 
Note that the first and last rows of $A$ are fixed and depended on $b$ while the other rows of $A$ can be determined in a great freedom. Since the energy $H(p,q)$
is a scalar, we can let these coefficients change depending on a scalar $\alpha\in \mathbb{R}$ 
in a particular chosen way. Of course, due to the degrees of freedom, such a choice may not be unique. Then, we hope to choose $\alpha$ such that the conditions \eqref{eq:conditiongroup1} are satisfied. Let's first take a closer look at classic examples and then introduce the general higher-order methods.

\subsection{New modified algorithms based on existing methods}

\subsubsection{$\alpha$-Rattle}\label{sec:SR}

In this section, we aim to propose a method based on the Shake-Rattle algorithm in section \ref{subsec:shakerattle}. As known, this method is two stage method with order $2$. If we perturb $b$, the order can be reduced.  Since $s=2$, with the conditions \eqref{eq:conditiongroup1}, there is only one degree of freedom left for choosing the parameters. 
In particular, we let
\begin{gather}\label{eq:bbcoeff}
b_1=1/2+\alpha, ~~b_2=1/2-\alpha,~~\alpha\in \mathbb{R}.
\end{gather}
Then, the matrix $A(\alpha)$-$\widehat{A}(\alpha)$ is fully determined, and the $\alpha$-PRK, which we call $\alpha$-Rattle, in the form Butcher table is given:
\begin{equation} \label{eq:PRKcoeff}
\begin{split}
\begin{array}{c|cc}
  0& 0 & 0\\
  1& \frac{1}{2}+\alpha & \frac{1}{2}-\alpha\\
  \hline
   &\frac{1}{2}+\alpha & \frac{1}{2}-\alpha\\
\end{array},\quad\quad~
\begin{array}{c|cc}
 \frac{1}{2}+\alpha & \frac{1}{2}+\alpha & 0\\
  \frac{1}{2}+\alpha & \frac{1}{2}+\alpha & 0\\
  \hline
   &\frac{1}{2}+\alpha & \frac{1}{2}-\alpha\\
\end{array}.
\end{split}
 \end{equation}
 We can see that this new $\alpha$-PRK method is different from the previous method given in \eqref{eq:2IIIA-IIIB}. For fixed $\alpha$, it not only can preserve the symplecticity and quadratic invariants, but also can guarantee the numerical solutions to fall on the manifold $\mathcal{M}$. When we adjust $\alpha$ each step, the method can preserve energy and quadratic invariants, and also to guarantee the numerical solutions to fall on the manifold $\mathcal{M}$. 
 
 We write out the details of the $\alpha$-Rattle method below for convenience:
\begin{equation} \label{eq:coeffPQ}
\begin{split}
\begin{cases}
&P_1=p_0+h(1/2+\alpha)\left(-H_{q}(P_1, Q_1)-G^{T}(Q_1)\Lambda_1  \right), \\
&P_2=P_1,\\
&Q_1=q_0,\\
&Q_2=q_0+h \left((\frac{1}{2}+\alpha)H_p(P_1, Q_1)+(\frac{1}{2}-\alpha)H_p(P_2, Q_2) \right),\\
&g(Q_1)=g(Q_2)=0,\\
\end{cases}
\end{split}
 \end{equation}
The numerical solution $(p_1, q_1)$ is then solved by
 \begin{equation} \label{eq:solvep2}
\begin{cases}
q_1 &=Q_2,\\
\ell_1&=-H_{q}(P_1, Q_1)-G^{T}(Q_1)\Lambda_1,\\
\ell_2&=-H_{q}(P_2, Q_2)-G^{T}(Q_2)\Lambda_2,\\
p_1&=p_0+h\left((\frac{1}{2}+\alpha)\ell_1+(\frac{1}{2}-\alpha)\ell_2 \right),\\
0&=G(q_1)H_{p}(p_1, q_1).
\end{cases}
 \end{equation}
The system \eqref{eq:coeffPQ} can be reduced to that 
\begin{equation} \label{eq:solvep2}
\begin{split}
\begin{cases}
&P=p_0-h(1/2+\alpha) H_{q}(P, q_0)-hG^{T}(q_0)\lambda_1, \\
&q_1=q_0+h \left((\frac{1}{2}+\alpha)H_p(P, q_0)+(\frac{1}{2}-\alpha)H_p(P, q_1) \right),\\
&0=g(q_1).\\
\end{cases}
\end{split}
 \end{equation}
where $\lambda_1:=(1/2+\alpha)\Lambda_1$.
There are three unknowns in this system.
Then, one can solve $(p_1, \lambda_2)$ by
 \begin{equation} \label{eq:solvep3}
\begin{cases}
p_1&=p_0-h(\frac{1}{2}+\alpha)H_{q}(P, q_0)-h(\frac{1}{2}-\alpha)H_{q}(P, q_1)
-hG^{T}(q_0)\lambda_1-hG^{T}(q_1)\lambda_2 ,\\
0&=G(q_1)H_{p}(p_1, q_1),
\end{cases}
 \end{equation}
where $\lambda_2:=(1/2-\alpha)\Lambda_2$.

\subsubsection{$\alpha$-PRK III methods}\label{sec:PRKIII}

We now construct the $\alpha$-PRK III methods from Lobatto IIIA or IIIB methods. For unconstrained Hamiltonian systems, this method has been used by Sun \cite{Sun1995} to construct symplectic PRK method. However, as we pointed out earlier, this method generally does not conserve the Hamiltonian energy. 

The general Runge-Kutta method with parameters $(c, A, b)$ for non-autonomous equations $\dot{y}=f(t,  y)$ is given by 
\[
y_1=y_0+h\sum_{j=1}^s b_j f(t_0+c_j h, k_j),~~~~
k_i=y_0+h\sum_{j=1}^s a_{ij}f(t_0+c_j h, k_j),~i=1,\ldots, s.
\]
The classical $B(p), C(\eta)$ and $D(\zeta)$ conditions for the RK method with parameters $(c, A, b)$ are given by
\begin{gather}
\begin{split}
B(p):&\quad \sum_{i=1}^s b_i c_i^{q-1}=\frac{1}{q}\quad q=1,\ldots, p;\\
C(\eta): &\quad \sum_{j=1}^s a_{ij}c_j^{q-1}=\frac{c_i^q}{q}
\quad i=1,\ldots, s,~q=1,\ldots, \eta;\\
D(\zeta): & \quad \sum_{i=1}^s b_i c_i^{q-1}a_{ij}=\frac{b_j}{q}(1-c_j^q)
\quad i=1,\ldots, s,~q=1,\ldots, \zeta.
\end{split}
\end{gather}
See \cite[IV.5, pp71]{Hairer2006} for more details on the $B(p),C(\eta)$ and $D(\zeta)$ assumptions for RK methods.

\begin{theorem}\label{thm:sun1995}
\cite{Sun1995}
Suppose that an $s$-stage RK method $(c,A,b)$
with distinct abscissae $c_{i}$ and $b_{i}\neq0$ satisfies
$B(p),C(\eta)$ and $D(\zeta)$. Then the PRK method generated by
$\left(c_{i},a_{ij},\hat{a}_{ij}=b_{j}(1-\frac{a_{ji}}{b_{i}}),
b_{i}\right)$ is symplectic and of order $q=\mathrm{min} (p,
2\eta+2, 2\zeta+2, \eta+\zeta+1)$ for Hamiltonian systems.
\end{theorem}
As a remark, the equations considered in this work are autonomous and the parameters $c$ do not appear in the scheme. However, the conditions $B(p), C(\eta), D(\zeta)$ still make sense: one may find the parameters $c$ from $C(\eta)$ condition and then require $B(p), D(\zeta)$ to hold.

Here we construct several specific parameterized $\alpha$-PRK III methods for Hamiltonian systems with holonomic constraints
following the idea described in Algorithm \ref{alg:equip}.

One method to parameterize the Lobatto IIIA or IIIB methods is to keep $b$ to be constant at first but to change $a_{i2}$ in some suitable ways. 
This leads to that 
 $3$-stage parameterized $\alpha$-Lobatto IIIA-III$\widehat{A}$  
\begin{equation} \label{eq:2alpha3IIIA-IIIA}
\begin{split}
\begin{array}{c|ccc}
  0& 0 & 0 & 0 \\
  \frac{1}{2}& \frac{5}{24}-\alpha&\frac{1}{3}-\alpha &2\alpha-\frac{1}{24}\\
  1& \frac{1}{6} & \frac{2}{3}&  \frac{1}{6}\\
  \hline
  & \frac{1}{6}& \frac{2}{3} &  \frac{1}{6}\\
\end{array},\quad\quad~
\begin{array}{c|ccc}
  0 & \frac{1}{6} & 4\alpha-\frac{1}{6} & 0 \\
  \frac{1}{2} & \frac{1}{6} &\frac{1}{3}+\frac{2}{3}\alpha &0\\
  1 & \frac{1}{6} & \frac{5}{6}-8\alpha & 0\\
  \hline
  & \frac{1}{6} & \frac{2}{3} &  \frac{1}{6}\\
\end{array}.
\end{split}
 \end{equation}
Or the $3$-stage parameterized $\alpha$-Lobatto III$\widehat{B}$-IIIB 
\begin{equation} \label{eq:2alpha3IIIB-IIIB}
\begin{split}
\begin{array}{c|ccc}
  0& 0 & 0 & 0 \\
  \frac{1}{2}& \frac{5}{24}-\frac{\alpha}{2}&\frac{1}{3}+\alpha &-\frac{1}{24}-\frac{\alpha}{2}\\
  1& \frac{1}{6} & \frac{2}{3}&  \frac{1}{6}\\
  \hline
  & \frac{1}{6}& \frac{2}{3} &  \frac{1}{6}\\
\end{array},\quad\quad~
\begin{array}{c|ccc}
 0 & \frac{1}{6} & 2\alpha-\frac{1}{6} & 0 \\
  \frac{1}{2} & \frac{1}{6} &\frac{1}{3}-\alpha &0\\
  1 & \frac{1}{6} & \frac{5}{6}+2\alpha & 0\\
  \hline
  & \frac{1}{6} & \frac{2}{3} &  \frac{1}{6}\\
\end{array}.
\end{split}
 \end{equation}

One can check that all the \tcb{$\alpha$-PRK III} methods used in \eqref{eq:2alpha3IIIA-IIIA}-\eqref{eq:2alpha3IIIB-IIIB} satisfy that $B(2)$ and $C(1)$. Hence, those four $\alpha$-PRK methods have second order accuracy according to Theorem \ref{thm:sun1995}.
We can construct general $s$-stage $\alpha$-PRK methods based on IIIA-IIIB methods in the same way. Note that those methods are different from the previous $\alpha$-PRK  based on $W$-transformation method given in \eqref{eq:3IIIA-IIIB}. When $\alpha=0$, all the $\alpha$-PRK methods presented in \eqref{eq:3IIIA-IIIB}, \eqref{eq:2alpha3IIIA-IIIA}-\eqref{eq:2alpha3IIIB-IIIB} are reduced to the classical Lobatto IIIA-IIIB pairs. 

\subsection{The general method and some theoretical results}
As we can see from the above example, we can first parameterize the coefficients $b(\alpha),A(\alpha)$, and then determine the entire $\alpha$-PRK method according to the conditions \eqref{eq:conditiongroup1}. 

\subsubsection{The general method and properites}

With the discussion above, we propose the method as follows:
\begin{algorithm}[H]
\caption{(Energy and quadratic invariant preserving method)}\label{alg:equip}
\begin{algorithmic}[1]
\State Choose a particular form of functions $b=b(\alpha)$, and $A=A(\alpha)$
such that there exists an interval $I\subset \mathbb{R}$ satisfying
\begin{itemize}
\item For all $1\le i\le s$, $b_i(\alpha)\neq 0$  for all most all $\alpha\in I$; $\sum_i b_i(\alpha)=1$.
\item $a_{1j}(\alpha)=0$, $a_{sj}(\alpha)=b_j(\alpha)$ for all $1\le j\le s$ and all $\alpha\in I$.
\item Determine $\widehat{b}$ and $\widehat{A}$ by relation \eqref{eq:conditiongroup1}.
\end{itemize}
\State Pick computational time $T>0$; time step $h>0$. Let $N=\lceil T/h \rceil$.
Pick initial data $(p_0, q_0)$.
\For {$n \text{ in } 1: N$}
\State  Let $(p_n(\alpha), q_n(\alpha))$ be given by \eqref{eq:PRK} with coefficients above and initial data $(p_{n-1}, q_{n-1})$.
\State Choose $\alpha^*\in I$ such that $H(p_1(\alpha^*), q_1(\alpha^*))=H(p_0, q_0)$.
\State Let $(p_n, q_n)=(p_n(\alpha^*), q_n(\alpha^*))$.
\EndFor
\end{algorithmic}
\end{algorithm}

From the algorithm, we can see that in order to preserve the constrained manifold $\cM$ we can impose the conditions $a_{1j}=0,~a_{sj}=b_j,$ or equivalently, $\hat{a}_{is}=0,~~\hat{a}_{i1}=b_1,~i=1,\ldots, s,$ and then to determine the symplectic PRK methods form the relation \eqref{eq:conditiongroup1}. 

For the general $s$-stage ($s\geq 3$) RK methods, due to the increase of degrees of freedom, the parameterization methods is not unique, and there may be many other methods. We now explore the properties of the new methods.

\begin{theorem}\label{thm:conservationproperty}
If the method in Algorithm \ref{alg:equip} has a solution, then the energy $H(p_n, q_n)$ is a constant. Moreover, for all  $D\in \mathscr{D}$, 
the quadratic invariant $q_n^TD p_n$ is also a constant.
\end{theorem}

\begin{proof}
The energy constant is obvious by the method. For the quadratic invariant, we note that Theorem \ref{thm:quadraticinvar} tells us that for any $\alpha$, $q_n^T(\alpha)Dp_n(\alpha)=q_0Dp_0$.
Hence, the claims follow. We skip the details.
\end{proof}

Here, as that for symplectic RK methods, let us point out the difference between the conservation of quadratic invariants and symplecticity.
As we have seen, for PRK methods with fixed coefficients, \eqref{eq:conditiongroup1} guarantees the symplecticity and conservation of quadratic invariants. In fact, for Hamiltonian system, conservation of quadratic invariants and symplecticity are equivalent. Now, if we apply the method for all initial data $(p_0, q_0)$ (all trajectories), the method given by Theorem \ref{thm:conservationproperty} may not be symplectic although \eqref{eq:conditiongroup1} is satisfied. The reason is that the coefficients in the method depend on the initial data now, and thus $dp_n\wedge dq_n$ may not be equal to $dp_0 \wedge dq_0$.
Compared to the fact that symplectic involves the differential of numerical solutions, the quadratic invariants depend only on the values of the numerical solutions. The conservation of quadratic invariants still holds by relations \eqref{eq:conditiongroup1}.
Based on the discussion above, our method can be applied in two possible cases.
\begin{enumerate}
\item We apply the method for all trajectories so that the energy and quadratic invariants are preserved 
for all trajectories. However, the method may not be symplectic though.

\item We fix down the coefficients obtained by a particular trajectory. Then, the method is symplectic, preserves the quadratic invariants, and preserves the energy for the chosen trajectory.
\end{enumerate}

\subsubsection{Existence of the numerical solutions and order of the methods}
Clearly, the method in Algorithm \ref{alg:equip} depends crucially on the solvability of 
\begin{gather}\label{eq:energyrequirement}
\mu(\alpha, h):= H(p_1(\alpha, h), q_1(\alpha, h))-H(p_0, q_0)=0,
\end{gather}
for small enough $h>0$ and some given consistent initial values $(p_0, q_0)\in \mathcal{Q}\times \mathcal{Q}$.
Clearly, $\mu(\alpha, h)$ is the error measuring the error for energy. As $h\to 0^+$, $\mu(\alpha, h)\to 0$ for all $\alpha$.
We assume the order of the energy satisfies the follows. 
\begin{assumption}\label{ass:methodconditions}
We assume the method for $\alpha=0$ is a special base method that satisfies the following.
\begin{enumerate}
\item Suppose $\mu(\alpha, h)$ is analytic in a neighborhood of $(0, 0)$. 
\item The order of energy approximation is higher for $\alpha=0$ than other $\alpha$ values near $0$. In other words, there exist $r, r_1$ with $r\ge r_1> 1$ such that
\begin{gather}\label{eq:asslimitcond}
\lim_{h\to 0^+}\frac{1}{h^r}\mu(0, h)\neq 0,\quad~\lim_{h\to 0^+}\frac{1}{h^{r_1}}\mu(\alpha, h)\neq 0.
\end{gather}
\end{enumerate}
\end{assumption}

\begin{remark}
To understand the second term in Assumption \ref{ass:methodconditions}, 
we note that
\[
\mu(\alpha, h)=H(p_1(\alpha, h), q_1(\alpha, h))
-H(p(h), q(h))
\]
where $(p(t), q(t))$ is the exact solution of the Hamiltonian system with initial data $(p_0, q_0)$. Hence, $\mu$ is in fact the numerical error of the energy for the scheme. The second iterm in Assumption \ref{ass:methodconditions} is the convergence order of the energy, which is often $O(h^{\beta+1})$ if the method has a convergence order $\beta$. 
\end{remark}

By the analyticity, one has in a neighborhood of $(0, 0)$ that
\[
\mu(\alpha, h)=\sum_{\ell\ge 0, m> 1}c_{\ell, m}\alpha^{\ell} h^m
\]
Since $\mu$ cannot be identically zero, there must some $\ell, m$ such that $c_{\ell,m}\neq 0$.
Hence, there always exists some $r(\alpha)>0$ such that
$\lim_{h\to 0^+}\mu(\alpha, h)/h^{r(\alpha)}$ exists and is nonzero. Hence, the second item in the assumption roughly says we have
\begin{gather}\label{eq:energydiffaux}
\mu(\alpha, h)=\left[ c_{0,r}h^{r}+\sum_{m>r}c_{0,m}h^m \right]+c_{kr_1}\alpha^k h^{r_1}+\cdots,~~\quad k>0, r\geq r_1>1.
\end{gather}

We further assume that the dependence in $\alpha$ is not degenerate, and without los of generalization in \eqref{eq:energydiffaux} for $k=1$.   
If $k\neq 1$, we may set $\tilde{\alpha}=\alpha^k$. Of course, in terms of $(\tilde{\alpha}, h)$, the analyticity of $\mu$ may break down. However, as soon as $\mu$ is $C^1$ in $\tilde{\alpha}$ and $h$, the claim below still holds. Provided $k=1$, we can prove that he method can give a sequence of numerical solutions, as below.
\begin{theorem}\label{thm:existence}
Suppose Assumption \ref{ass:methodconditions} holds and $r, r_1$ are the numbers as in Assumption \ref{ass:methodconditions}. If it holds that 
\begin{gather}\label{eq:solvecond}
\lim_{h\to 0^+}\frac{1}{h^{r_1}}\frac{\partial \mu}{\partial \alpha}(0,0)\neq 0,
\end{gather}
then the equation \eqref{eq:energyrequirement} has a solution $\alpha^*$ for $h$ sufficiently small, and this solution can be written as 
\begin{gather}
\alpha^*(h)=\eta(h)h^{r-r_1}
\end{gather}
for some smooth function $\eta(\cdot)$, with $\lim_{h\to 0^+}\eta(h)\neq 0$. 
\end{theorem}

\begin{proof}
The proof can be done by the implicit function theorem, following similar approach as in \cite{BrugnanoEnergy}.
Define a new variable
$\eta=h^{-(r-r_1)}\alpha.$
Then,
\[
\tilde{\mu}(\eta, h)=h^{-r}\mu(h^{(r-r_1)}\eta, h).
\]
By the assumptions, $\tilde{\mu}$ is analytic in a neighborhood of $(0, 0)$.
Moreover,
\[
\frac{\partial \tilde{\mu}}{\partial \eta}(0,0)\neq 0,
\]
by \eqref{eq:solvecond}.
Clearly, $\tilde{\mu}=0$
implies $\mu=0$. By implicit function theorem,
$\tilde{\mu}=0$ implicitly defines a smooth function
$\eta=\eta(h)$
for $h\in J$, where $J$ is a small interval containing $0$.
By Assumption \ref{ass:methodconditions}, it is not hard to see (depending on $c_{0,r}$ and $c_{1,r_1}$ in \eqref{eq:energydiffaux})
$\eta(0)\neq 0.$
The claims then follow if we define $\alpha^*=h^{r-r_1}\eta$.
\end{proof}

With the relation $\alpha^*=\eta(h) h^{r-r_1}$, the convergence order of the method given in Algorithm \ref{alg:equip} is in fact can be the same as that for $\alpha=0$. This means that the order of convergence can be preserved. We in fact have the following claim.
\begin{corollary}\label{cor:order}
Suppose that for all $\alpha\neq 0$ near $0$, the order of PRK is $r_1-1$ and that the order of PRK scheme for $\alpha=0$ is $r-1$. If the conditions in Theorem \ref{thm:existence} hold, then the order of convergence for the method given in Algorithm \ref{alg:equip} is equal to $r-1$.
\end{corollary}

\begin{proof}
By the assumptions,
\[
\|p_1(\alpha, h)-p(h)\|=\eta_1(\alpha, h)h^{r_1},
~~\|q_1(\alpha, h)-q(h)\|=\eta_2(\alpha, h)h^{r_1},
\]
where $(p(h), q(h))$ are the exact solutions for the Hamiltonian system with initial data $(p_0, q_0)$. Here, $\eta_1(\alpha, h)$
and $\eta_2(\alpha, h)$ are bounded smooth functions. 
Then, it is easy to see $|\eta_i(\alpha, h)-\eta_i(0, h)|\le C |\alpha|$ for $i=1,2$. Hence, with $\alpha=\alpha^*$
\[
\|p_1(\alpha, h)-p_1(0,h)\|\le Ch^{r_1}h^{r-r_1}=Ch^r.
\]
The same is true for $q_1$. 
By the assumption on the accuracy for $\alpha=0$, the claim follows.
\end{proof}

As a direct application of the Theorem \ref{thm:existence} and Corrollary \ref{cor:order}, we in the next try to find more explicit conditions corresponding to the condition in Theorem \ref{thm:existence} to hold (namely
$\lim_{h\to 0^+}h^{-r_1}\partial_{\alpha}\mu \neq 0$) and determine $r_1$ for the modified Shake and Rattle algorithm described in section \ref{sec:SR}. We first of all assume the following.

\begin{assumption}
\label{ass:independentconstraints}
The matrix $G(q)$ is of rank $m$ for all $q$, and $H_{pp}(p,q)$ is invertible for all $(p, q)$.
\end{assumption}

Clearly, the first is reasonable as we assume the $m$ holonomic constraints are independent. 
Below, we perform the discussion with Assumption \ref{ass:independentconstraints}.
Motivated by the condition in Theorem \ref{thm:existence}, we define the operator
\begin{gather}
\cD_sf:=\lim_{h\to 0^+} \Big[ h^{-s}\partial_{\alpha}f(\alpha, h)|_{\alpha=0}\Big].
\end{gather}
Since $\lim_{h\to 0^+}p_1(0, h)=p_0$ and $\lim_{h\to 0^+}q_1(0, h)=q_0$, one has
\begin{gather}
\cD_{s}\mu=(\cD_s p_1)\cdot H_p(p_0, q_0)
+(\cD_s q_1)\cdot H_q(p_0, q_0).
\end{gather}

Now, we aim to find the minimum $r$ suh that $\cD_r\mu \neq 0$.
First consider $s=1$. By \eqref{eq:solvep2} and $P(0,h)\to p_0$, one has
\[
\cD_1 P=-H_q(p_0, q_0)
-\frac{1}{2}(\cD_0 P)\cdot H_{pq}(p_0, q_0)-G^T(q_0)\cdot \cD_0\lambda_1.
\]
Clearly, $\cD_0 P=0$. Hence,
$\cD_1P=-H_q(p_0, q_0)-G^T(q_0)\cdot \cD_0\lambda_1.$
Consequently,
\[
\begin{split}
\cD_1 q_1 &=H_p(p_0, q_0)-H_p(p_0, q_0)+\frac{1}{2}[\cD_0 P \cdot H_{pp}(p_0, q_0)
+\cD_0 P H_{pp}(p_0, q_0)+\cD_0 q_1 H_{qp}(p_0, q_0)]\\
&=0.
\end{split}
\]
Similarly,
\[
\begin{split}
\cD_1 p_1 &=[-H_q(p_0,q_0+H_q(p_0, q_0))]
-G^T(q_0)\cD_0(\lambda_1+\lambda_2)
-\cD_0q_1\cdot \nabla_{qq}g(q_0)\cdot \lim_{h\to 0^+}\lambda_2(0,h)\\
&= -G^T(q_0)\cD_0(\lambda_1+\lambda_2).
\end{split}
\]

By the constraint of $p_1$ in \eqref{eq:solvep3},
$0=\cD_0 q_1\otimes H_p(p_0,q_0):\nabla^2g(q_0)+G(q_0)\cdot [\cD_1 p_1\cdot H_{pp}(p_0, q_0)]=G(q_0)\cdot[\cD_1 p_1\cdot H_{pp}(p_0, q_0)].$
Since the $m$ constraints are independent and $H_{pp}$ is invertible, $GH_{pp}G^T$ is invertible and thus
$\cD_0(\lambda_1+\lambda_2)=0.$
Hence,
$\cD_1p_1=0.$

Now, we move to $s=2$. By the $q_1$ equation in \eqref{eq:solvep2}, 
\begin{equation*}
\begin{split}
\cD_2q_1 &=\lim_{h\to 0^+} h^{-1}[H_p(P, q_0)-H_p(P, q_1)]
+\cD_1P\cdot H_{pp}(p_0, q_0)+\frac{1}{2}\cD_1q_1\cdot H_{qp}(p_0, q_0)\\
&=- H_p(p_0, q_0)\cdot H_{qp}(p_0, q_0)
-[H_q(p_0, q_0)+G^T(q_0)\cdot\cD_0\lambda_1]\cdot H_{pp}(p_0, q_0)
\end{split}
\end{equation*}
By the constraint in \eqref{eq:solvep2}, one has
$G(q_0)\cD_2q_1=0$. Using this, one can determine $\cD_0\lambda_1$.

Similarly, we have
\begin{equation*}
\begin{split}
\cD_2p_1&=\lim h^{-1}(H_q(P, q_1)-H_q(P, q_0))
-\cD_1P\cdot H_{pq}(p_0, q_0)-\frac{1}{2}\cD_1q_1\cdot H_{qq}-
G^T(q_0)\cD_1[\lambda_1+\lambda_2]-\\
&\cD_1q_1\cdot\nabla_2g(q_0)\cdot\lim\lambda_2 \\
&=H_p(p_0, q_0)\cdot H_{qq}(p_0, q_0)
+[H_q(p_0, q_0)+G^T(q_0)\cD_0\lambda_1]\cdot H_{pq}(p_0, q_0)
-G^T(q_0)\cD_1[\lambda_1+\lambda_2].
\end{split}
\end{equation*}
Moreover, By the constraint of $p_1$ in \eqref{eq:solvep3},
we have
\[
0=H_p(p_0, q_0)\otimes \cD_2q_1:\nabla_{qq}g(q_0)
+(\cD_2p_1\cdot H_{pp}(p_0, q_0)
+\cD_2q_1\cdot H_{qp}(p_0, q_0))\cdot\nabla_qg(q_0)
\]
With this, one can determine $\cD_1(\lambda_1+\lambda_2)$.
Hence, we can find that
\begin{gather}
\cD_2\mu=H_p\otimes H_p:H_{qq}-H_q\otimes H_q:H_{pp}
+R,
\end{gather}
where
\[
R:= [G^T\cD_0\lambda_1]\otimes H_p: H_{pq}
-H_p G^T \cD_1[\lambda_1+\lambda_2]
-[G^T\cD_0\lambda_1]\otimes H_q: H_{pp}.
\]
Every term in $R$ involves $G$, while the first two terms 
in $\cD_2\mu$ do not depend on $G$. Hence,
$\cD_2\mu\neq 0$ for all most all $(p, q)$.
It is possible that at some special points, this can be zero.

\begin{remark}
For a more straightforward derivation, one may take derivatives on $\alpha$ in \eqref{eq:solvep2}-\eqref{eq:solvep3} and then expand 
the quantities like $\partial_{\alpha}P, \partial_{\alpha}p_1$ etc in terms of powers of $h$. Then compare the coefficients, one can derive the same formulas.
\end{remark}

To summarize, we expect that for almost all $(p_0, q_0)$,
\[
\lim_{h\to 0^+}\frac{1}{h^2}\frac{\partial \mu(0, h)}{\partial \alpha}\neq 0.
\]
This coincides with the fact the method for $\alpha\neq 0$ is first order.
Hence, the method should be second order by Corollary \ref{cor:order}.
We know the order of convergence for $\alpha=0$ is two or $r=3$.
Hence, mostly, we will have
\[
\alpha^*=O(h).
\]
It could be possible that at some special points, $\lim_{h\to 0^+}\frac{1}{h^2}\frac{\partial \mu(0, h)}{\partial \alpha}=0$.  For these points, there could still be solutions for $\alpha^*$ but one may have larger $\alpha^*\gg h$, as we need to look into higher orders.

\section{Numerical Experiments}
\label{sec:numer}
In this section, we present numerical examples to verify our theoretical findings. We verify that the method constructed in this paper can accurately conserve the quadratic invariants and constraint manifolds of the system, and at the same time, the appropriate parameters $\alpha^*$ can be selected to make the system energy conservation at each step. As far as we know, this is the first time that there is a numerical method for a Hamiltonian constrained  systems that can hold all three quantities at once. We will also verify the order of convergence of the parameterized method and prove that the parameter perturbation will not decrease the order of convergence of the numerical solutions.
All the experiments in this paper are implemented using Julia on Version 1.4.2
and Packages: NLsolve v4.4.0.

 \subsection{Spherical pendulum}

 \begin{figure}[!ht]
\begin{center}
$\begin{array}{cc}
 \includegraphics[scale=0.35]{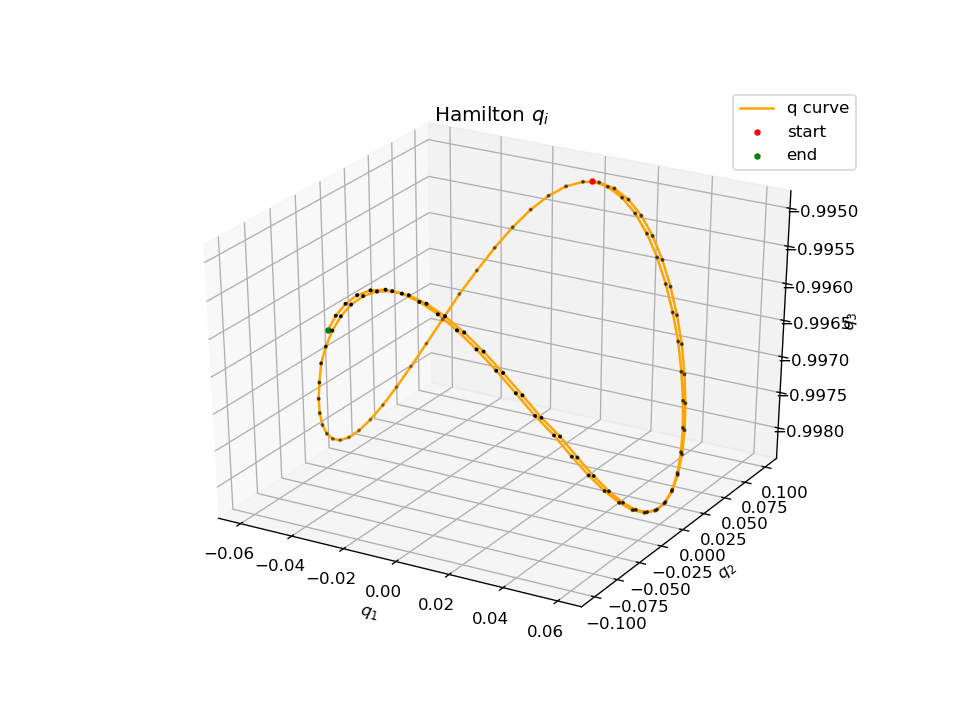}
 \includegraphics[scale=0.35]{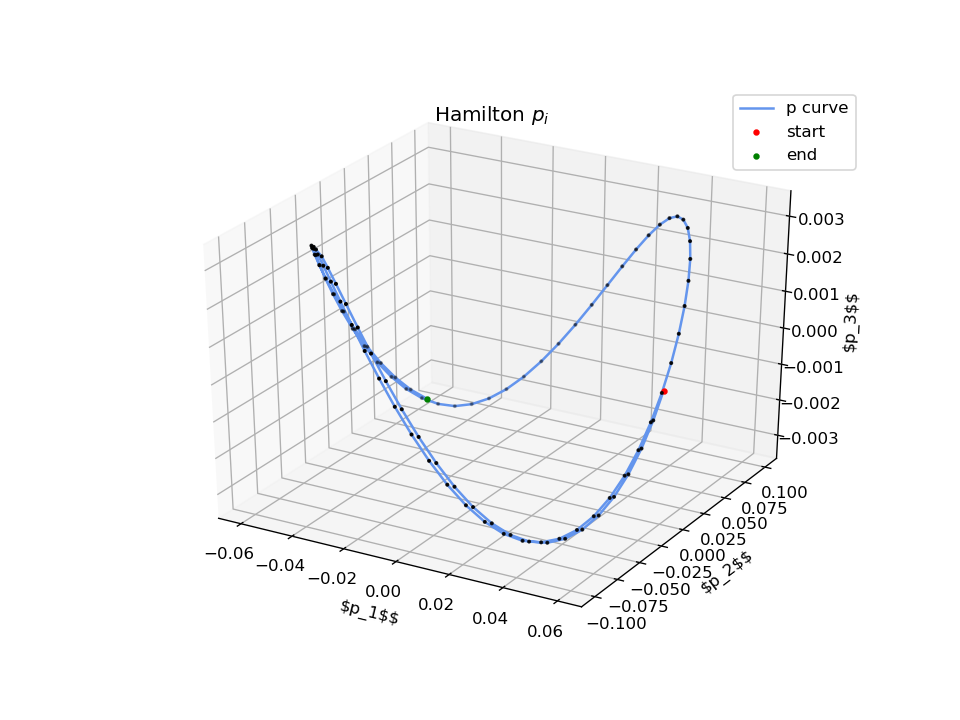}
 \end{array}$
\end{center}
\caption{Numerical solutions.}
\label{Piceigen}
\end{figure}

\begin{figure}[!ht]
\begin{center}
$\begin{array}{cc}
 \includegraphics[scale=0.4]{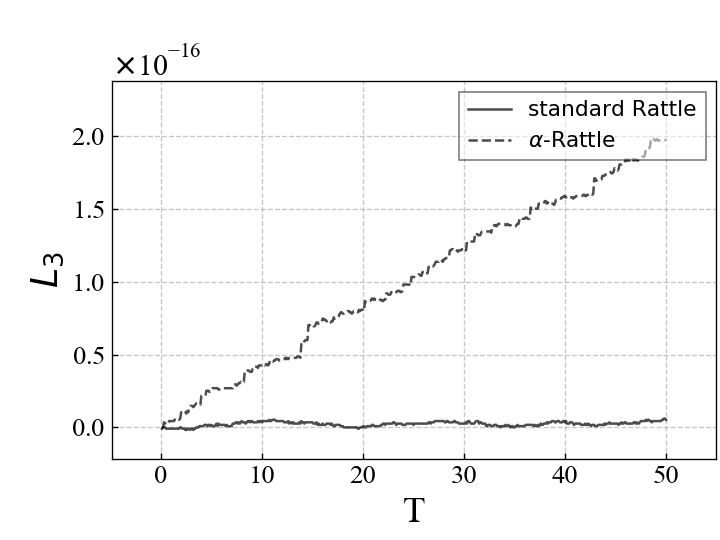}
 \quad\quad
 \includegraphics[scale=0.4]{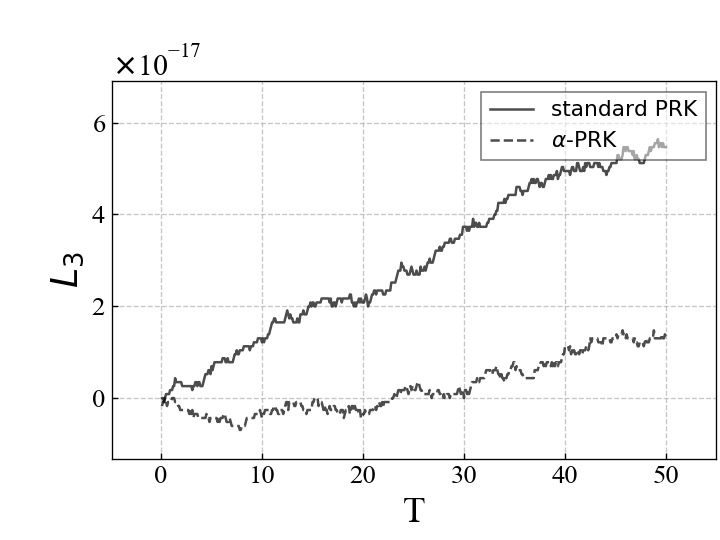}
 \end{array}$
\end{center}
\caption{The errors in the third component of the angular momentum for standard Rattle and Lobatto IIIA-IIB methods and the $\alpha$-Rattle and $\alpha$-PRK III methods.}
\label{ang}
\end{figure}

\begin{figure}[htbp]\centering

\subfigure[Energy errors for standard Rattle method.]{
\begin{minipage}[t]{0.45\linewidth}
\centering
\includegraphics[width=2.5in]{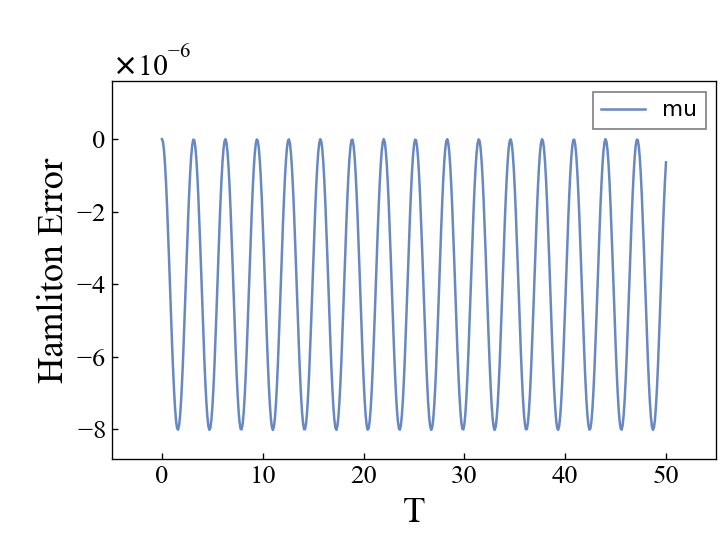}
\end{minipage}
}
\subfigure[Energy errors for standard Lobatto IIIA-IIIB method.]{
\begin{minipage}[t]{0.45\linewidth}
\centering
\includegraphics[width=2.5in]{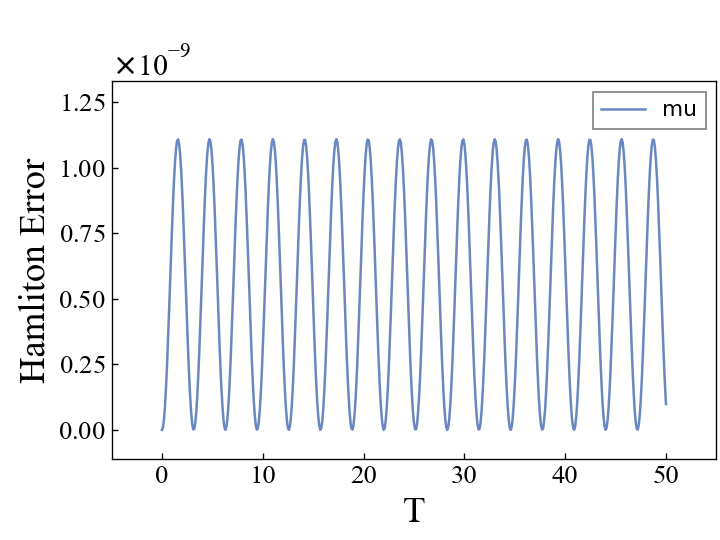}
\end{minipage}
}                 

\subfigure[$\alpha^{*}$ for $\alpha$-Rattle method.]{
\begin{minipage}[t]{0.45\linewidth}
\centering
\includegraphics[width=2.5in]{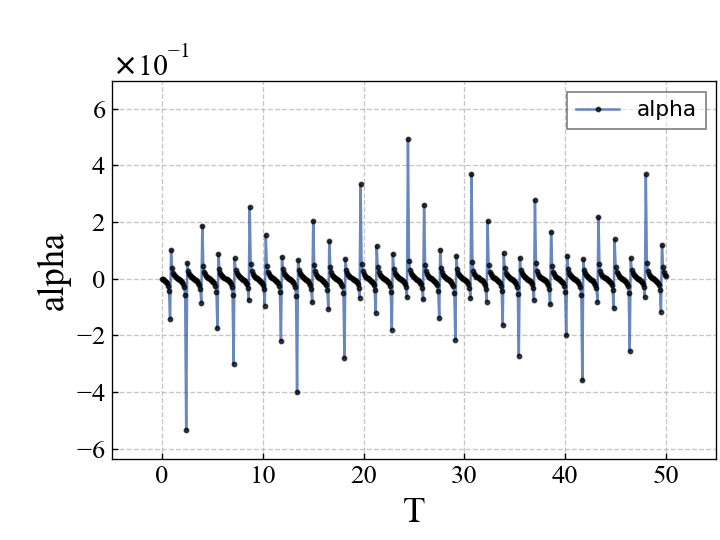}
\end{minipage}
}
\subfigure[Energy errors for $\alpha$-Rattle method.]{
\begin{minipage}[t]{0.45\linewidth}
\centering
\includegraphics[width=2.5in]{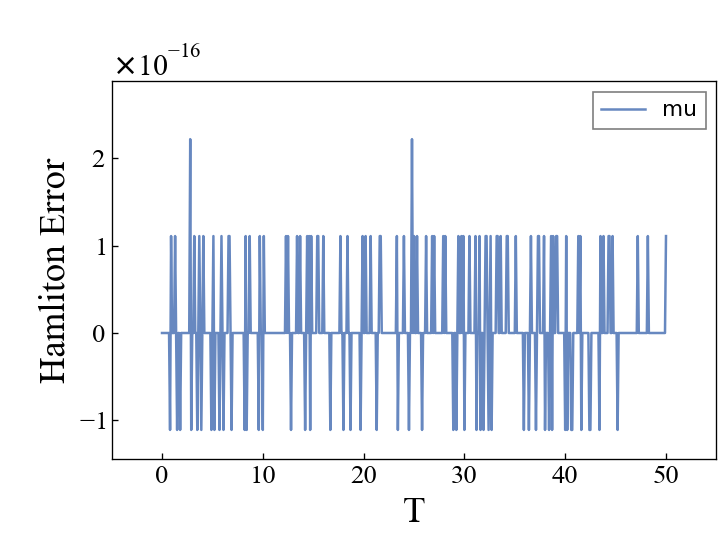}
\end{minipage}
}                 

\subfigure[$\alpha^{*}$ for $\alpha$-PRK III method.]{
\begin{minipage}[t]{0.45\linewidth}
\centering
\includegraphics[width=2.5in]{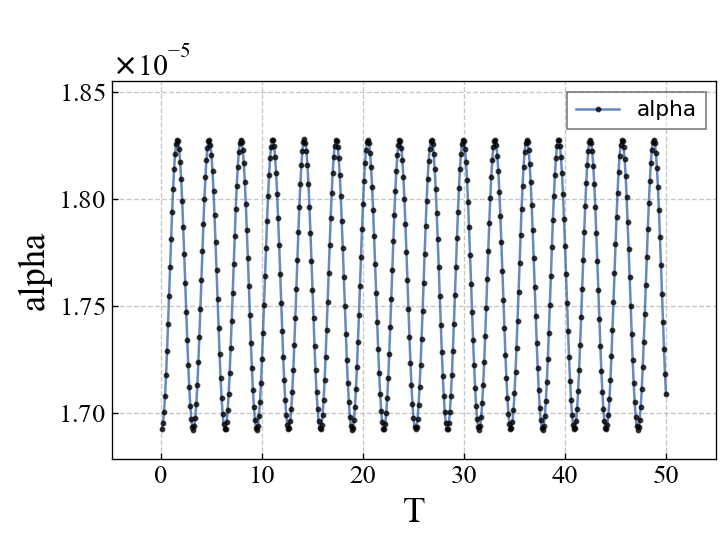}
\end{minipage}
}
\subfigure[Energy errors for $\alpha$-PRK III method.]{
\begin{minipage}[t]{0.45\linewidth}
\centering
\includegraphics[width=2.5in]{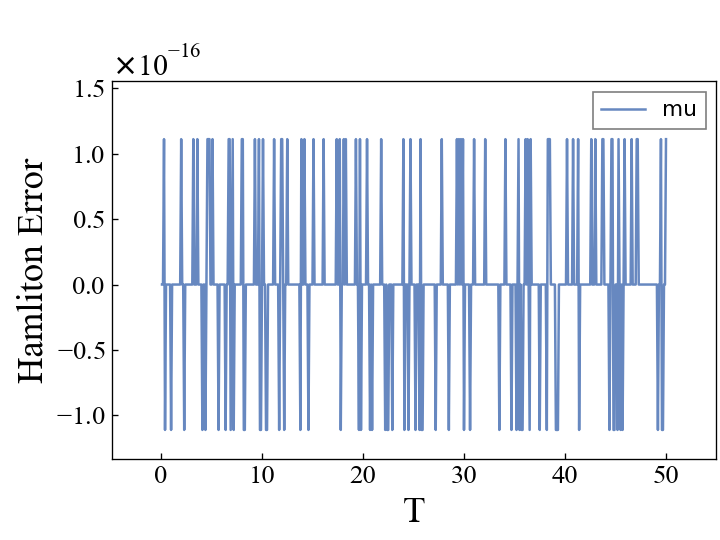}
\end{minipage}
}                 
\caption{ $\alpha^{*}$ and energy errors with $h=0.1$.}
\label{figex1}
\end{figure}

Consider the spherical pendulum equation with $H(p,q)=\frac{1}{2}(p_1^2+p_2^2+p_3^2)+q_3$ and $g(q)=(q_1^2+q_2^2+q_3^2)-1$, so that
\begin{equation} \label{eq26}
\begin{split}
H_p(p,q)=\begin{bmatrix}
p_1   \\
p_2  \\
p_3 \\
\end{bmatrix}, ~
H_q(p,q)=\begin{bmatrix}
0   \\
0  \\
1 \\
\end{bmatrix}, ~
G^{T}(q)=2 \begin{bmatrix}
q_1   \\
q_2  \\
q_3 \\
\end{bmatrix}.
\end{split}
 \end{equation}
For this example, the Lagrangian
\begin{gather*}
L(Q, \dot{Q})=\frac{1}{2}|\dot{Q}|^2-Q_3,
\end{gather*}
 is invariant under rotations about $z$ axis. In particular, taking $D_1$ to be a $2\times 2$ anti-symmetric matrix and setting
\begin{gather}
D:=\left[\begin{array}{cc}
D_1 & 0\\
0 &0
\end{array}\right].
\end{gather}
Then, using $D^T=-D$, and $(e^{sD}v)_3=v_3$, it can be verified easily that
\begin{gather*}
g(e^{sD}q)=q^T e^{sD^T}e^{sD}q-1=q^Tq-1=0,~~~
L(e^{sD}\dot{q}, e^{sD}q)=\frac{1}{2}|e^{sD}\dot{q}|^2
-(e^{sD}q)_3=L(\dot{q}, q).
\end{gather*}
Hence, the quantity $I(p, q)=q^T Dp$ is invariant under this Hamiltonian dynamics. In particular, this implies that the third component of the angular momentum $ L(p,q)=q\times p$ is conservative. Physically, the forces (force due to the constraint which points along the radius of the sphere, the gravity), $q$ and $z$-axis are coplanar, so $z$-component of the torque is zero. Thus, the $z$-component of the angular moment is conserved.  We will verify through numerical experiments in this example that the parameterized method we constructed can preserve both the Hamilton function and $L_3$.

\begin{table}[ht] 
\caption{The errors $e_{p}$, $e_{q}$ and
the orders at $T=0.5$ for the $\alpha$-Rattle method for Example 1}
\begin{center}\label{tab1}
\begin{tabular}{ccccc}
\hline $h $ & $e_{p}$ & \mbox{order} & $e_{q}$ & \mbox{order}\\
\hline 
0.25 & 3.3643e-4         &- & 3.5220e-4 & -\\
0.125 & 8.6813e-5       &1.9543 &8.9671e-5 & 1.9737\\
0.0625& 2.1895e-5       &1.9872 &2.2535e-5 & 1.9924\\
0.03125& 5.4863e-6    &1.9967 &5.6416e-6 & 1.9980\\
0.015625& 1.7323e-6    &1.9992 &1.4108e-6 & 1.9995\\
\hline
\end{tabular}
\end{center}

\caption{The errors $e_{p}$, $e_{q}$ and
the orders at $T=1$ for the $\alpha$-Rattle method for Example 1}
\begin{center}\label{tab2}
\begin{tabular}{ccccc}
\hline $h $ & $e_{p}$ & \mbox{order} & $e_{q}$ & \mbox{order}\\
0.1 & 1.1216e-4         &- &1.0774e-4 & -\\
0.05 & 4.5419e-5       &1.3042 &4.8323e-5 & 1.1569\\
0.025& 1.5315e-5       &1.5683 &2.0176e-5 & 1.2600\\
0.0125& 1.6383e-6       &3.2247 &1.5628e-6 & 3.6905\\
0.00625& 5.7032e-7    &1.5226 &5.9045e-7 & 1.4043\\
\hline
\end{tabular}
\end{center}

\caption{The errors $e_{p}$, $e_{q}$ and
the orders at $T=1$ for the $\alpha$-PRK III method for Example 1}
\begin{center}\label{tab3}
\begin{tabular}{ccccc}
\hline $h $ & $e_{p}$ & \mbox{order} & $e_{q}$ & \mbox{order}\\
\hline 
0.25 & 4.0025e-7         &- & 4.7611e-7 & -\\
0.125 & 2.5089e-8       &3.9957 &2.9843e-8 & 3.9958\\
0.0625& 1.5692e-9       &3.9989 &1.8665e-9 & 3.9989\\
0.03125& 9.8091e-11    &3.9998 &1.1666e-10 & 3.9998\\
0.015625& 5.9634e-12    &4.0399 &6.9571e-12 & 4.0677\\
\hline
\end{tabular}
\end{center}
\end{table}

We take the initial value $q_0 =(0, \sin(0.1), -\cos(0.1)), p_0 =(0.06, 0, 0)$ for the simulations.  In Fig. \ref{ang},  we present the third component of the angular momentum $L(p,q)$. From which we can see that both the standard Rattle and Lobatto IIIA-IIIB methods and the parameterized methods constructed in this paper can preserve this quadratic invariants to machine precision, that is up to $10^{-16}$.

In Fig. \ref{figex1}, we show the results for the chosen parameters $\alpha$ and computed Hamiltonians for the $\alpha$-Rattle in \eqref{eq:solvep2}-\eqref{eq:solvep3} and $\alpha$-PRK III method in \eqref{eq:2alpha3IIIA-IIIA}. Clearly, the standard symplectic methods (i.e. Rattle and Lobatto IIIA-IIIB methods) have energy error range from $10^{-9}$ to $10^{-6}$, while the parameterized symplectic methods can make the energy conservation in the sense of that the energy error reach the machine accuracy.  At each step, we can find parameters $\alpha^{*}=\alpha(h, y_{n})$ that achieve energy conservation. We can see from $(c)$ in Fig. \ref{figex1} that the parameter is small in most cases for $\alpha$-Rattle method, but it may take a slightly larger value at a very few points, and such an exceptional points may lower the order of the parameterized methods. 

In the Table \ref{tab1} and \ref{tab2} we give the specific order of convergence for the $\alpha$-Rattle method. We can see that the order of Table 1 is completely consistent with the standard algorithm of second order when computed on $[0,0.5]$, while the order of Table 2 is slightly reduced when we computed on $[0, 1]$. We found that there is some $\alpha$ that is big as shown in Figure \ref{figex1} (c) between $t=0.5$ and $t=1$, and this seems to reduce the computational accuracy as commented above. We are not clear with the deeper reason for this. Maybe, the conditions in Theorem \ref{thm:existence} might be violated at some point, or maybe the nonlinear system is not solved with high accuracy as commented below in Remark \ref{rmk:nonlinearsyssolve}.  
For the 3 stage $\alpha$-PRK method, we can see from $(e)$ in Figure \ref{figex1}  that the parameter is uniform small about $O(10^{-5})$. As shown in Table \ref{tab3}, the convergence order of $\alpha$-PRK III method is preserved exactly as the standard 3 stage PRK method. That is forth order. 

\begin{remark}\label{rmk:nonlinearsyssolve}
In the specific implementation of the parameterized methods presented in this article, it is very important to accurately solve the corresponding nonlinear system of equations (like \eqref{eq:coeffPQ}-\eqref{eq:solvep2}). If the accuracy of the numerical solutions is not enough (for example, the constraints are not satisfied to enough accuracy), the accuracy of the energy computation might be reduced to like from $10^{-8}$ to $10^{-12}$, and the searched parameters $\alpha^*$ might not be very good.  One may try to improve the accuracy of nonlinear equations by giving a good initial value during the iteration. For example, use the numerical results of the standard PRK methods as the initial value in the iteration of the nonlinear equations for the corresponding parameterized methods. For future work, one may introduce some correction field or augmented variables to make the nonlinear system more stable to solve.
\end{remark}

\subsection{Satellites system}

\begin{figure}[!ht]
\begin{center}
$\begin{array}{cc}
 \includegraphics[scale=0.45]{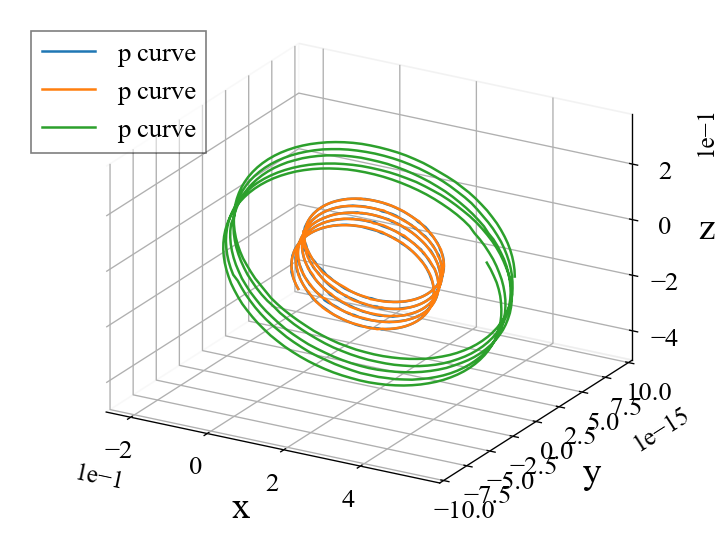}
 \includegraphics[scale=0.45]{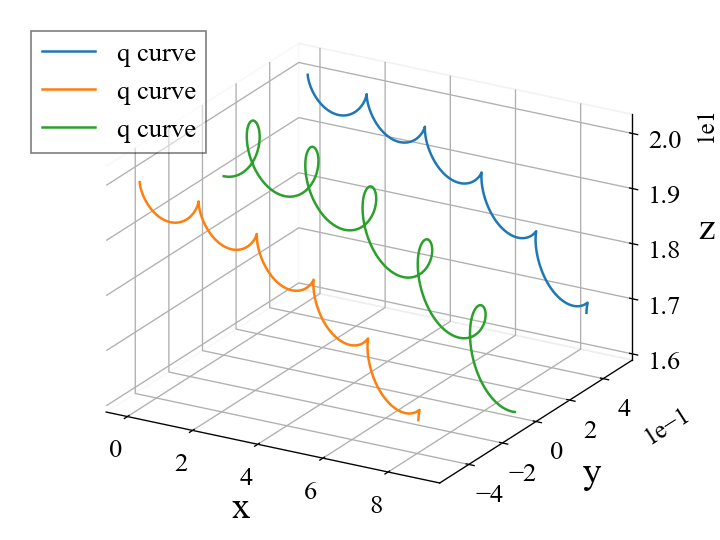}
 \end{array}$
\end{center}
\caption{Numerical solutions obtained by $\alpha$-Rattle method with $h=0.1$.}
\label{Piceigen}
\end{figure}

We discuss a closed-loop rotating triangular tethered satellites system, including three satellites (considered as mass-points) of masses $m_i, i = 1, 2, 3,$ joined by inextensible, tight, and massless tethers, of lengths $L_i, i = 1, 2, 3$, respectively. For sake of simplicity, we assume unit masses and lengths, and normalize the gravity constant. Consequently, if $q_i :=(x_i,y_i,z_i)^T\in\mathbb{R}^3, i =1,2,3$, are the positions of the three satellites, the constraints $0=g(q)$ are given by
 \begin{equation} \label{eq:constrants2}
\begin{split}
\begin{cases}
&g_1(q)=(q_1-q_2)^T (q_1-q_2)-1=0, \\
&g_2(q)=(q_2-q_3)^T (q_2-q_3)-1=0,\\
&g_3(q)=(q_3-q_1)^T (q_3-q_1)-1=0.\\
\end{cases}
\end{split}
 \end{equation}  The Hamiltonian is specified by 
 $H(p,q)=\sum_{i=1}^3\left(\frac{1}{2}p_i^{T}p_i- \frac{1}{\sqrt{q_i^{T}q_i}}\right).$

 Hence, we have that 
 \begin{equation} \label{eq:Hfunction}
\begin{split}
&H_p=p=\begin{bmatrix}
p_1   \\
p_2  \\
p_3 \\
\end{bmatrix}, ~
H_q=f(q):=
\begin{bmatrix}
 (q_1^{T}q_1)^{-\frac{3}{2}} q_1  \\
 (q_2^{T}q_2)^{-\frac{3}{2}} q_2  \\
 (q_3^{T}q_3)^{-\frac{3}{2}} q_3 \\
\end{bmatrix},~
G^{T}(q)=
 2 \begin{bmatrix}
q_1-q_2 & 0  & q_1-q_3 \\
q_2-q_1 & q_2-q_3 & 0\\
0 & q_3-q_2  & q_3-q_1\\
\end{bmatrix}.
\end{split}
 \end{equation}

\begin{figure}[htbp]\label{ex1}
\centering

\subfigure[Energy errors for Rattle method with $h=0.1$.]{
\begin{minipage}[t]{0.45\linewidth}
\centering
\includegraphics[width=2.5in]{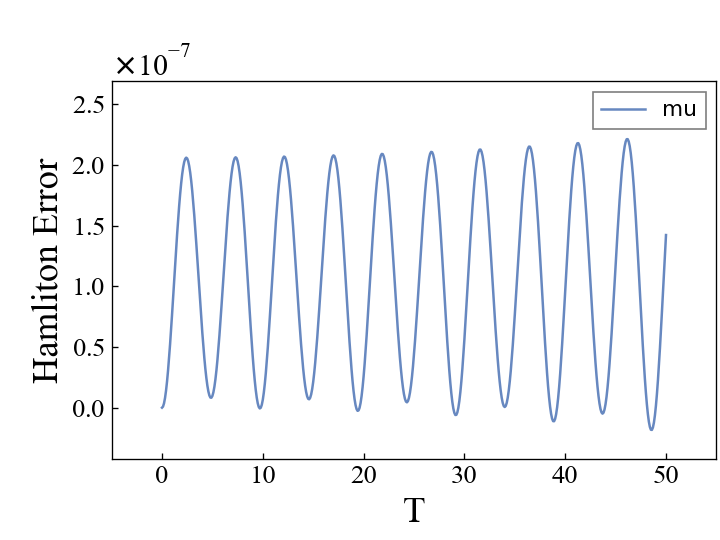}
\end{minipage}
}
\subfigure[Energy errors for Lobatto IIIA-IIIB method with $h=0.2$.]{
\begin{minipage}[t]{0.45\linewidth}
\centering
\includegraphics[width=2.5in]{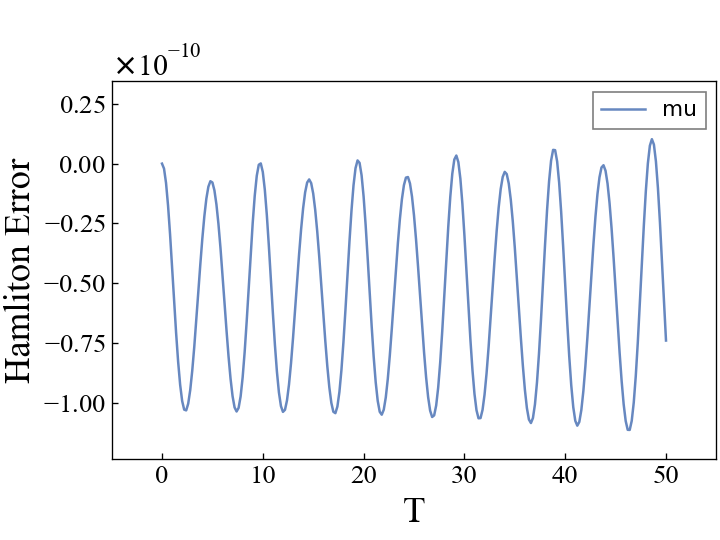}
\end{minipage}
}                

\subfigure[$\alpha^{*}$ for $\alpha$-Rattle method with $h=0.1$.]{
\begin{minipage}[t]{0.45\linewidth}
\centering
\includegraphics[width=2.5in]{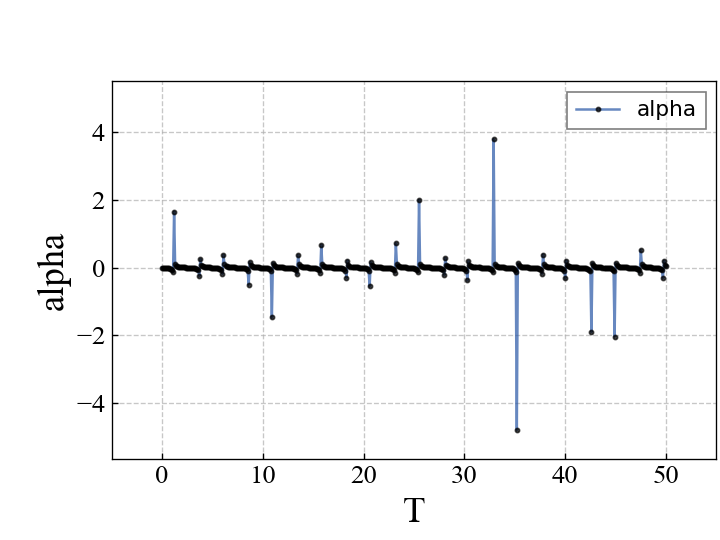}
\end{minipage}
}
\subfigure[Energy errors for $\alpha$-Rattle method with $h=0.1$.]{
\begin{minipage}[t]{0.45\linewidth}
\centering
\includegraphics[width=2.5in]{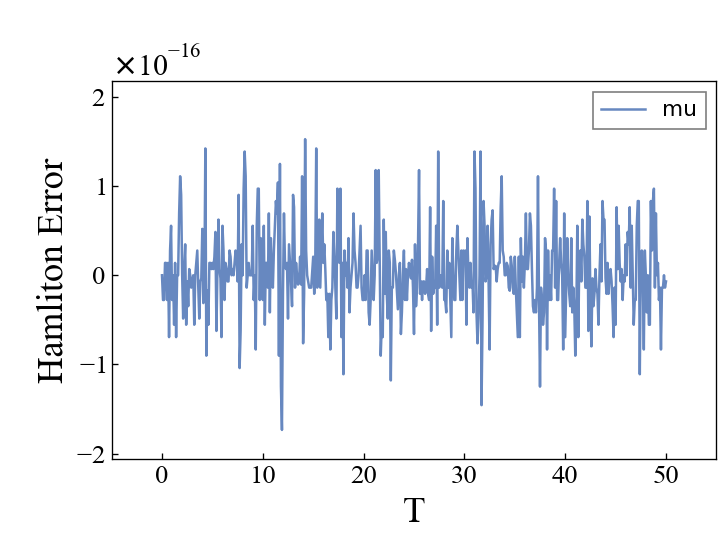}
\end{minipage}
}%

\subfigure[$\alpha^{*}$ for $\alpha$-PRK III method  with $h=0.2$.]{
\begin{minipage}[t]{0.45\linewidth}
\centering
\includegraphics[width=2.5in]{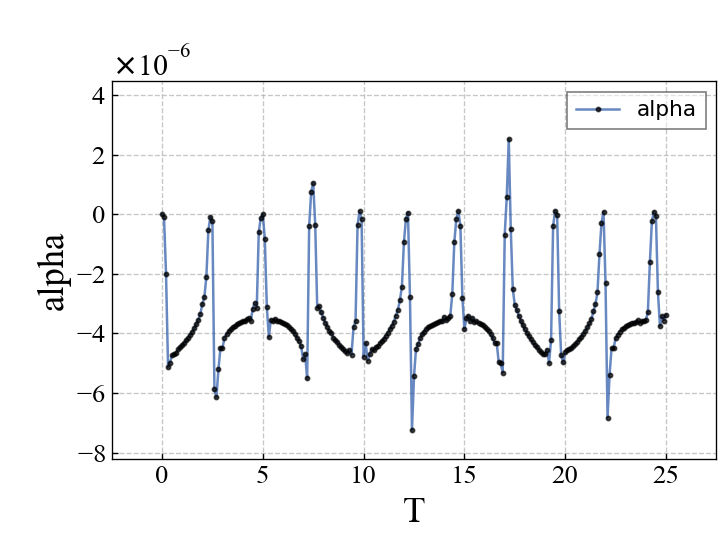}
\end{minipage}
}
\subfigure[Energy errors for $\alpha$-PRK III method  with $h=0.2$.]{
\begin{minipage}[t]{0.45\linewidth}
\centering
\includegraphics[width=2.5in]{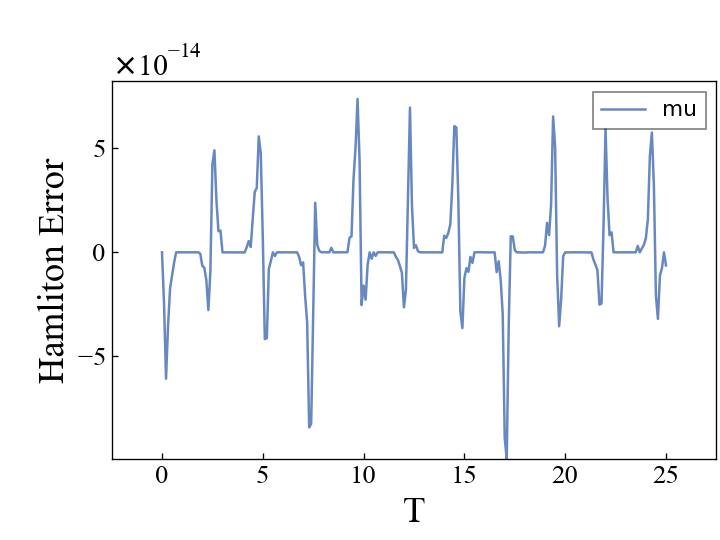}
\end{minipage}
}%
\caption{ $\alpha^{*}$ and energy errors}
\end{figure}

For the simulation, we choose the consistent initial values given by 
 \begin{equation*} 
\begin{split}
q_1(0)=\begin{bmatrix}
0   \\
\frac{1}{2}  \\
z_0 \\
\end{bmatrix}, ~
q_2(0)=\begin{bmatrix}
0   \\
-\frac{1}{2}  \\
z_0 \\
\end{bmatrix}, ~
q_3(0)=\begin{bmatrix}
0   \\
0  \\
z_0-\frac{\sqrt{3}}{2} \\
\end{bmatrix}, ~
p_1(0)=p_2(0)=\begin{bmatrix}
0   \\
0 \\
0 \\
\end{bmatrix}, ~
p_3(0)=\begin{bmatrix}
v_0   \\
0  \\
0 \\
\end{bmatrix}.
\end{split}
 \end{equation*}
 where $z_0 = 20$ and $v_0$ is such that the initial Hamiltonian is zero. 
This provides a configuration in which the first two
satellites remain parallel to each other, moving in the planes $y = 1/2$ and $y = -1/2$, respectively, 
and the third one moves around the tether joining the first two, in the plane $y = 0$.

\begin{table}[ht]
\caption{The errors $e_{p}$, $e_{q}$and
the orders at $T=1$ for the $\alpha$-Rattle method for Example 2}
\begin{center}\label{tab4}
\begin{tabular}{ccccc}
\hline $h $ & $e_{p}$ & \mbox{order} & $e_{q}$ & \mbox{order}\\
\hline 
0.25 & 1.2290e-3         &-           & 1.9300e-3 & -\\
0.125 & 3.0460e-4       &2.0125 &4.7835e-4 & 2.1247\\
0.0625& 7.59745e-5    &2.0033 &1.1931e-4 & 2.0033\\
0.03125& 1.8968e-5    &2.0018 &2.9789e-5 & 2.0018\\
0.015625& 4.7270e-6  &2.0046 &7.4235e-6 & 2.0046\\
\hline
\end{tabular}
\end{center}

\caption{The errors $e_{p}$, $e_{q}$and
the orders at $T=1$ for the $\alpha$-PRK III method for Example 2}
\begin{center}\label{tab5}
\begin{tabular}{ccccc}
\hline $h $ & $e_{p}$ & \mbox{order} & $e_{q}$ & \mbox{order}\\
\hline 
0.25 & 1.2299e-6         &-           & 1.9307e-6 & -\\
0.125 & 7.1884e-8       &4.0967 &1.1284e-7 & 4.0967\\
0.0625& 3.0717e-9      &4.5485 &4.8219e-9 & 4.5485\\
0.03125& 9.7230e-11   &4.9815 &1.5263e-10 & 4.9815\\
0.015625& 6.4283e-12 &3.9188 &1.0091e-11 & 3.9188\\
\hline
\end{tabular}
\end{center}
\end{table}

For this more complex example, our numerical results are similar to the first example. The standard symplectic PRK methods have energy errors, but the parameterized methods can preserves energy to machine accuracy. At each step, we can find the parameter values that make the energy conserved. Meanwhile, as shown in Table \ref{tab4} and \ref{tab5}, the order of convergence of the two parameterized methods can be completely consistent with the standard methods. By comparison, it is found that $\alpha$-PRK III methods are often superior to the $\alpha$-Rattle method, and the parameter values of $\alpha$-PRK III methods are usually uniformly small, while $\alpha$-Rattle method may have a few points with large $\alpha^*$ values.

\section*{Acknowledgements} 
The authors would like to thank Mr. Changze Chen (Northwest University) for his great help and support in the implementation of the numerical experiments in this paper. The work of L. Li was partially sponsored by NSFC 11901389, Shanghai Sailing Program 19YF1421300 and NSFC 11971314. 
The work of D. Wang was partially sponsored by NSFC 11871057, 11931013 and Project for Young Science and Technology Star of Shaanxi Province in China (2018 KJXX-070).

\bibliographystyle{alpha}
\bibliography{Hamilmanifold}

\end{document}